\documentclass[10pt]{article}
\usepackage{amsthm,amsfonts,amsmath,amscd,amssymb,epsfig}
\usepackage[all]{xy}
\usepackage{enumerate,array}
\title{The role of string topology\\ in symplectic field theory}
\author{K.~Cieliebak and J.~Latschev}
\date{5 October 2007}
\theoremstyle{plain}
\newtheorem{theorem}{Theorem}[section]
\newtheorem{thm}[theorem]{Theorem}

\newtheorem{cor}[theorem]{Corollary}

\newtheorem{prop}[theorem]{Proposition}
\newtheorem{lemma}[theorem]{Lemma}
\newtheorem{conj}[theorem]{Conjecture}

\newtheorem{thmintro}{Theorem}

\theoremstyle{remark}

\theoremstyle{definition}
\newtheorem*{definition}{Definition}
\newcommand{\wt}{\widetilde}

\newcommand{\p}{\partial}

\newcommand{\om}{\omega}

\newcommand{\eps}{\varepsilon}
\newcommand{\pHi}{\varphi}

\newcommand{\K}{{\mathbb{K}}}

\newcommand{\Z}{{\mathbb{Z}}}
\newcommand{\R}{{\mathbb{R}}}
\newcommand{\C}{{\mathbb{C}}}
\newcommand{\Q}{{\mathbb{Q}}}

\renewcommand{\d}{{\bf d}}

\renewcommand{\a}{{\bf a}}

\newcommand{\h}{{\bf h}}

\newcommand{\F}{{\bf F}}
\newcommand{\G}{{\bf G}}

\newcommand{\A}{{\bf A}}
\newcommand{\B}{{\bf B}}

\newcommand{\D}{{\bf D}}

\renewcommand{\H}{{\bf H}}
\renewcommand{\L}{{\bf L}}
\newcommand{\PHI}{{\mathbf \Phi}}

\newcommand{\const}{{\rm const}}

\newcommand{\sgn}{{\rm sgn\,}}

\newcommand{\CZ}{{\rm CZ}}

\newcommand{\length}{{\rm length}}

\newcommand{\SFT}{{\rm SFT}}
\newcommand{\lin}{{\rm lin}}

\newcommand{\MM}{\mathcal{M}}
\newcommand{\CC}{\mathcal{C}}

\newcommand{\HH}{\mathcal{H}}

\renewcommand{\AA}{\mathcal{A}}
\newcommand{\PP}{\mathcal{P}}
\newcommand{\WW}{\mathcal{W}}
\newcommand{\DD}{\mathcal{D}}
\newcommand{\comment}[1]{}
\newcommand{\x}{\times}
\newcommand{\chains}{{\mathcal{C}}}

\newcommand{\one}{{\mathbf 1}}
\newcommand{\str}{{\rm string}}
\begin{document}
\maketitle

\section{Introduction}\label{sec:intro}

Over the past two years, we have developed a program for incorporating
holomorphic curves with boundary on some closed Lagrangian submanifold
into the general framework of symplectic field theory (SFT). These
notes give an outline of the
proposed theory, with an emphasis on ideas, geometric intuition and
a description of the resulting algebraic structures. Precise proofs
involve a substantial amount of technical work and will appear
elsewhere~\cite{CL:exact, CL:string}. 

The starting point of our work is the observation, originally due to
Fukaya~\cite{Fu:06}, that the new phenomena in the compactification of
moduli spaces that occur for curves with boundary are usefully
described in terms of the string topology of Chas and
Sullivan~\cite{CS}. Fukaya's work concerns the case of holomorphic
disks with boundary on some closed Lagrangian submanifold in a closed
symplectic manifold. He viewed these moduli spaces as ordinary or
equivariant chains on the free loop space of the Lagrangian
submanifold via the evaluation at the boundary. Boundary bubbling is
then modelled by the loop bracket or string bracket, depending on
whether or not the disks involved come with marked points
on the boundary.  

Our initial goal was to apply the same idea to compare the SFT
invariants of a unit cotangent bundle $S^*Q$ to the string 
topology of the underlying smooth manifold $Q$. It turns out that,
when suitably interpreted, the structures on the two sides are
isomorphic. The simplest statement of this isomorphism is discussed in
Section~\ref{sec:cot} of these notes. On the SFT side, it
involves the {\em linearized contact homology} $HC^\lin(T^*Q)$ of the
cotangent bundle $T^*Q$. Such a linearization in fact exists for any
exact symplectic manifold with a convex end. Moreover, one can derive the
following result from the algebraic properties of SFT as described
in~\cite{EGH}: 
\begin{thmintro}[cf.~Theorem~\ref{thm:filling_linear}]\label{thm:A}
Every exact symplectic manifold $(X,d\lambda)$ with convex boundary
gives rise to a linearized contact homology $HC^\lin(X,\lambda)$ which
has the structure of an involutive graded Lie bialgebra. 
\end{thmintro}
A precise definition of this concept is given in Section~\ref{sec:alg}.
It might be interesting to note that the bracket and cobracket are
defined in terms of curves of genus 0, but the proof of the involutivity
property involves moduli spaces of curves of genus 1. We also point
out that on the chain level one has the ``infinity version'' of this
structure. The homotopy theory of these chain-level structures is the 
subject of joint work with K.~Fukaya~\cite{CFL}.

In the context of Theorem~\ref{thm:A}, suppose further that $(X,d\lambda)$
contains an exact Lagrangian submanifold $Q \subset X$ which is
oriented and relatively spin. Consider its {\em string
homology} $H(\Sigma Q,Q)$, i.e. the singular homology (with
rational coefficients) of the space $\Sigma Q= C^0(S^1,Q)/S^1$ of
strings (= loops up to rotation in the domain) relative to the
constant strings. Chas and Sullivan~\cite{CS2}  showed that it too
carries the structure of a graded involutive Lie bialgebra. Now we have
\begin{thmintro}[cf.~Theorem~\ref{thm:filling_mor}]\label{thm:B}
The moduli spaces of holomorphic disks with one puncture and boundary
on $Q \subset (X,d\lambda)$ give rise to a chain map inducing a
morphism of involutive Lie bialgebras
$$
HC^{\lin}(X,\lambda) \to H(\Sigma Q, Q).
$$
\end{thmintro}
Again we remark that one in fact expects a chain-level version
of this theorem, giving a morphism of the underlying ``infinity
versions'' of the structure. However the precise formulation of this
chain level statement is complicated by the fact that the string
topology operations are only partially defined, and so far there does
not exist a precise statement of the invariance properties of the
``full structure'' they generate. 

In any case, specializing to the case of cotangent bundles,
i.e. $X=T^*Q$ and $Q \subset X$ the zero section and using the fact
that the above morphism respects suitable filtrations on both sides, one
obtains the following 
\begin{thmintro}[cf.~Theorem~\ref{thm:iso}]\label{thm:C}
For each closed oriented manifold $Q$, there is an isomorphism of
involutive graded Lie bialgebras
$$
HC^\lin(T^*Q,\lambda) \stackrel{\cong}{\longrightarrow} H(\Sigma Q,Q).
$$
\end{thmintro}
This result should be viewed as an $S^1$-equivariant
analogue of the isomorphism 
\begin{equation}\label{eq:floer}
   HF_*(T^*Q) \cong H_*(\Lambda Q)
\end{equation}
between the Floer homology of the cotangent bundle $T^*Q$ and the
singular homology of the free loop space $\Lambda Q=Map(S^1,Q)$,
intertwining the pair-of-pants product on $HF_*(T^*Q)$ with the loop 
product on $H_*(\Lambda Q)$ (see~\cite{Vi}, \cite{SW}, \cite{AS}). 
Note, however, that the isomorphism of Theorem C is more intricate, as
it respects a much more elaborate algebraic structure. At present, there is 
no known interesting bialgebra structure on the non-equivariant homology 
of the free loop space.

The string topology Lie bialgebra in the case when $Q$ is a surface is
the highly non-trivial structure discovered earlier by Goldman and
Turaev. In Section~\ref{sec:cot}, we illustrate Theorem~\ref{thm:C} by
translating known properties of this Lie bialgebra into existence
results for holomorphic curves with suitable asymptotics in the
symplectization of $S^*Q$.

One of the main points of these notes is that the above results
naturally fit into a more general theory of holomorphic curves with
boundary on some closed Lagrangian submanifold $Q \subset X$ which is
oriented and relatively spin (but not necessarily exact) in a general
symplectic cobordism $X$ between 
contact manifolds $(Y^\pm,\lambda^\pm)$. The main result of this
theory will be a master equation for a suitable generating function $\L$
for the moduli spaces of holomorphic curves with boundary on $Q$ which
directly generalizes the corresponding equation for the potential $\F$
of the symplectic cobordism known from SFT:
\footnote{Strictly speaking, at the time of writing this theorem is
  still conjectural, cf. technical remarks in
  Sections~\ref{sec:master} and \ref{sec:cob}.} 
\begin{thmintro}[cf.~Theorem~\ref{thm:master-L2}]
The potential $\L$ of a (not necessarily exact) pair $(X,Q)$ satisfies
the master equation
\begin{equation}
   (\p + \Delta + \hbar\nabla)(e^\L) =
   e^\L\overleftarrow{\H^+}-\overrightarrow{\H^-}e^\L, 
\end{equation} 
where $\Delta$ and $\nabla$ are operations in string topology
described in Section~\ref{sec:string}.  
\end{thmintro}
As we discuss in Section~\ref{sec:cob}, this result can be viewed as a
generalization of Fukaya's equation~\cite{Fu:06}
$$
\p \a_1 + \frac 1 2 [\a_1,\a_1]_\str = 0
$$
for the analogous generating function $\a_1$ counting holomorphic
disks with boundary on $Q\subset X$ (and without punctures). 
In order to use this result to construct new invariants of Lagrangian
submanifolds, one needs to understand chain-level string topology. Our
approach to this problem will be the subject of \cite{CL:string}.

Theorem D also leads to a precise formulation
(cf.~Theorem~\ref{thm:twisted-iso}) of the
principle that ``holomorphic curves with punctures asymptotic to a
Lagrangian submanifold $Q$ carry the same information as holomorphic
curves with boundary on $Q$''. 

One of the big open problems in the developement of SFT (besides the
analytic foundations) is the formulation of its relative version.
We believe that one geometrically very appealing solution to this
problem can be given with the help of string topology. Our ideas in
this direction, which were developed in discussions with K.~Mohnke,
are sketched in \ref{sec:rel}.  

We finish this introduction with some remarks about the structure of
these notes. To bring out the main ideas more clearly, we have tried
to avoid the discussion of technicalities in the main text as much as
possible. At the end of each section we then list some ``technical
remarks'' that outline more precisely the current status of the
results stated and mention the main technical issues that need
to be resolved in order to give complete proofs. 
Section~\ref{sec:SFT} starts with a quick
review of the algebraic formalism of SFT as described in~\cite{EGH}. 
In Section~\ref{sec:master} we then give a heuristic derivation of the
master equation with Lagrangian boundary
conditions. Section~\ref{sec:string} makes this somewhat more precise
by describing the necessary background from string topology. In
Section~\ref{sec:alg} we present some algebraic concepts, and in 
Section~\ref{sec:geo} we discuss how these structures can be used to describe 
both SFT and string topology. Section~\ref{sec:cot} discusses the
special case of cotangent bundles. Up to this point, we consider a
simplified setup, where all symplectic cobordisms and Lagrangian
submanifolds are assumed exact. In Section~\ref{sec:cob} we describe
the changes needed in the general case. Finally, in \ref{sec:rel},
which is joint work with K.~Mohnke, we outline our
approach to relative SFT. 

{\bf Acknowledgements. } 
In the course of this work, we profited in various degrees from discussions
with other mathematicians. The idea of relating SFT to string topology arose
at a conference in Stare Jablonki in the summer of 2004, where
K.C. attended a talk of K.~Fukaya and subsequently discussed the case 
of cylindrical contact homology with him.
We also want to thank
F.~Bourgeois,
A.~Cattaneo, 
T.~Ekholm,
Y.~Eliashberg, 
H.~Hofer,
D.~Indelicato, 
K.~Mohnke,
K.~Ono,
and D.~Sullivan 
for stimulating conversations.

\section{Symplectic field theory}\label{sec:SFT}

In this section we recall the Weyl formalism of SFT from~\cite{EGH}. 
Here, as in most of this paper (with the exception of
Section~\ref{sec:cob}), we work in the {\em exact category}, by which
we mean the following: \\
(i) All symplectic cobordisms $(X,\om)$ and Lagrangian submanifolds
$Q\subset X$ are {\em exact}, i.e.~$\om=d\lambda$ for a 1-form
$\lambda$ restricting to a contact form on the boundary and
$\lambda|_Q$ is exact. \\
(ii) All symplectic cobordisms $(X,\om)$ and Lagrangian submanifolds
have vanishing first Chern class resp.~Maslov class. Moreover, all
closed Reeb orbits $\gamma:S^1\to Y^\pm$ are nondegenerate and there
exist trivializations of $\gamma^*TX$ with respect to which all
relative first Chern classes vanish. \\ 
(iii) All Lagrangian submanifolds are oriented and relatively spin.

For example, these assumptions are satisfied in the following two
cases:

(1) $X=T^*Q$ is the cotangent bundle of an oriented manifold and $Q$
    the zero section. Here condition (i) is clear for the canonical
    1-form, $c_1(T^*X)=0$, and $Q$ is relatively spin. The rest of
    condition (ii) follows by choosing trivializations of $T(T^*Q)$
    induced by trivializations of $TQ$ along loops. 

(2) $X$ is a simply connected Stein manifold with $c_1=0$ and $Q$ is
    an oriented, spin, exact Lagrangian submanifold with vanishing
    Maslov class. Such manifolds arise e.g.~as vanishing cycles in
    fibres of exact Lefschetz fibrations over the disk (in which case
    $Q\cong S^n$). 

{\bf Contact manifolds. }
Let $(Y^{2n-1},\lambda)$ be a closed contact manifold. Use the
trivializations provided by the exactness hypothesis to define the
{\em Conley-Zehnder index} $\CZ(\gamma)\in\Z$ of a closed Reeb orbit
$\gamma:S^1\to Y$. Recall that the $k$-th iterate $\bar\gamma^k$ of a
simple closed Reeb orbit $\bar\gamma$ is called {\em good} if
$\CZ(\bar\gamma^k)\equiv \CZ(\bar\gamma)$ mod $2$. To each good closed
Reeb orbit $\gamma$ associate formal variables $p_\gamma,q_\gamma$
with gradings  
$$
   |p_\gamma| := n-3-\CZ(\gamma),\qquad |q_\gamma|:=n-3+\CZ(\gamma). 
$$
We denote by $\kappa_\gamma$ the multiplicity of the orbit $\gamma$.
Let $J$ be a cylindrical almost complex structure on $\R\times V$
adjusted to $\lambda$.
For ordered collections of closed Reeb orbits
$\Gamma^\pm=(\gamma_1^\pm,\dots,\gamma_{s^\pm}^\pm)$ and an integer
$g\geq 0$ 
denote by 
$$
  \MM_g(\Gamma^-,\Gamma^+)
$$ 
the moduli space of connected $J$-holomorphic curves of genus $g$ in
$\R\times Y$ with $s^+$ positive and $s^-$ negative punctures
asymptotic to the $\gamma_i^+$ resp.~$\gamma_j^-$. The exactness
hypothesis ensures that  
the following dimension formula holds:
$$
   \dim\MM_g(\Gamma^-,\Gamma^+) = (n-3)(2-2g-s^+-s^-) +
   \sum_i\CZ(\gamma_i^+) - \sum_j\CZ(\gamma_j^-). 
$$
If this dimension is $1$ denote by $n_g(\Gamma^-,\Gamma^+)\in\Q$ the
algebraic count of points in $\MM_g(\Gamma^-,\Gamma^+)/\R$. 
We introduce the formal variables
$$
p := \sum_\gamma \frac 1 {\kappa_\gamma} p_\gamma \gamma ,\qquad
q := \sum_\gamma \frac 1 {\kappa_\gamma} q_\gamma \gamma
$$
and the correlator
$$
{}^{-1}\langle
\underbrace{q,\dots,q}_{s^-};\underbrace{p,\dots,p}_{s^+}\rangle_g
:= \sum_{|\Gamma^\pm|=s^\pm} n_g(\Gamma^-,\Gamma^+)
\,q^{\Gamma^-}p^{\Gamma^+}, 
$$
and set
$$
\H_g := \sum_{s^-,s^+} \frac 1 {s^-!s^+!} {}^{-1}\langle
\underbrace{q,\dots,q}_{s^-};\underbrace{p,\dots,p}_{s^+}\rangle_g.
$$
Finally, let $\hbar$ be another formal variable of degree
$$
   |\hbar|=2(n-3)
$$
and define the {\em Hamiltonian}
$$
   \H := \frac{1}{\hbar}\sum_{g=0}^\infty \H_g\hbar^g. 
$$ 
The grading conventions are such that $\H$ is homogeneous of degree
$-1$.  

Let $\WW$ be the graded Weyl algebra of power series in the variables
$\hbar,p_\gamma$ with coefficients polynomial in the variables 
$q_\gamma$. $\WW$ is equipped with the associative
product $*$ in which all variables super-commute according to their
gradings except for the variables $p_\gamma,q_\gamma$ corresponding to
the same Reeb orbit $\gamma$, for which we have 
$$
   p_\gamma*q_\gamma - (-1)^{|p_\gamma||q_\gamma|}q_\gamma*p_\gamma =
   \kappa_\gamma\hbar. 
$$
Let $\PP$ be the graded commutative algebra of power series in the
variables $p_\gamma$ with coefficients polynomial in the variables 
$q_\gamma$. Using the commutation relations, every element $F\in\WW$
can be uniquely written in the {\em standard form}
\begin{equation}\label{eq:standard}
   F = \sum_{\Gamma,g}f_{\Gamma,g}(q)p^\Gamma\hbar^g
\end{equation}
with polynomials $f_{\Gamma,g}$. This gives us an identification $\WW
\cong \PP[[\hbar]]$ {\em as vector spaces}. From now on we will use this
identification, so all elements of $\WW$ are 
assumed to be written in the standard form \eqref{eq:standard}.

%
The following theorem encodes the boundaries of
$2$-dimensional moduli spaces via gluing of holomorphic curves, see
Figure~\ref{fig:1}.
\begin{figure}[h]
\begin{center}
\epsfbox{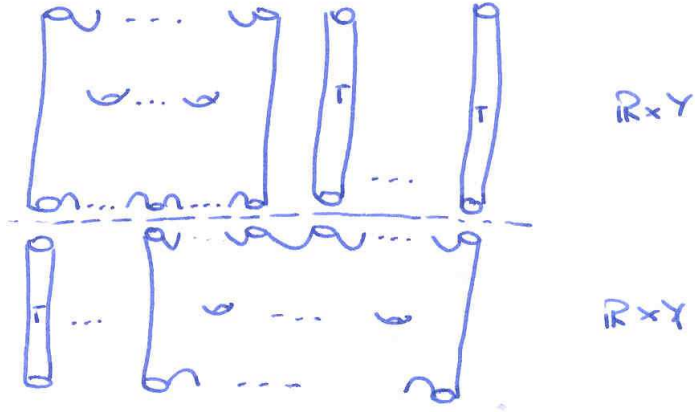}
\caption{Breaking of curves in a symplectization}
\label{fig:1}
\end{center}
\end{figure}
\begin{thm}[\cite{EGH}]\label{thm:master-H}
The Hamiltonian $\H\in\frac{1}{\hbar}\WW$ satisfies the {\em master
  equation} 
\begin{equation}\label{eq:master-H}
   \H*\H = 0.
\end{equation}
\end{thm}

{\bf Symplectic cobordisms. }
Next let $(X^{2n},\om)$ be an exact symplectic
cobordism with convex end $(\R_+\times Y^+,\lambda^+)$ and concave end
$(\R_-\times Y^-,\lambda^-)$. Let $J$ be a compatible almost complex
structure on $X$ that is cylindrical and adjusted to $\lambda^\pm$ on
the ends. For ordered collections of closed Reeb orbits 
$\Gamma^\pm=(\gamma_1^\pm,\dots,\gamma_{s^\pm}^\pm)$ in $Y^\pm$ and an
integer $g\geq 0$ denote by
$$
   \MM_{g}(X;\Gamma^-,\Gamma^+)
$$
the moduli space of connected $J$-holomorphic curves of genus $g$ in
$X$ with $s^+$ positive and $s^-$ negative punctures asymptotic to the
$\gamma_i^+$ resp.~$\gamma_j^-$. Under our exactness hypothesis, this
moduli space has expected dimension  
$$
   \dim\MM_{g}(X;\Gamma^-,\Gamma^+) = (n-3)(2-2g-s^+-s^-) +
   \sum_i\CZ(\gamma_i^+) - \sum_j\CZ(\gamma_j^-).  
$$
If this dimension is $0$ denote by $n_g(X;\Gamma^-,\Gamma^+)\in\Q$ the
algebraic count of points in $\MM_g(X;\Gamma^-,\Gamma^+)$. 
Define the correlators
$$
{}^{0}\langle
\underbrace{q,\dots,q}_{s^-};\underbrace{p,\dots,p}_{s^+}\rangle_g
:= \sum_{|\Gamma^\pm|=s^\pm} n_g(X;\Gamma^-,\Gamma^+)
\,q^{\Gamma^-}p^{\Gamma^+}, 
$$
where the sum is taken over all ordered collections
$\Gamma^\pm$. Define the {\em potential} $\F$ of the cobordism $X$ by 
$$
   \F:=\frac{1}{\hbar}\sum_{g=0}^\infty\F_g\hbar^g,
$$
where
$$
   \F_g := \sum_{s^-,s^+}\frac 1 {s^-!s^+!} {}^{0}\langle
   \underbrace{q,\dots,q}_{s^-};
   \underbrace{p,\dots,p}_{s^+}\rangle_g.   
$$
Denote by $\DD$ the space of formal power series in $\hbar$ and the
$p_\gamma^+$ with coefficients polynomial in the $q_\gamma^-$. 
Elements in $\WW^\pm$ act as differential operators from the
right/left on $\DD$ via the replacements 
$$
   q_\gamma^+\mapsto \kappa_\gamma \hbar 
       \overleftarrow{\frac{\p}{\p p_\gamma^+}},\qquad 
   p_\gamma^-\mapsto \kappa_\gamma \hbar 
       \overrightarrow{\frac{\p}{\p q_\gamma^-}}.
$$

The following theorem encodes the boundaries of 1-dimensional moduli
spaces via gluing of holomorphic curves, see Figure~\ref{fig:2}.
\begin{figure}[h]
\begin{center}
\epsfbox{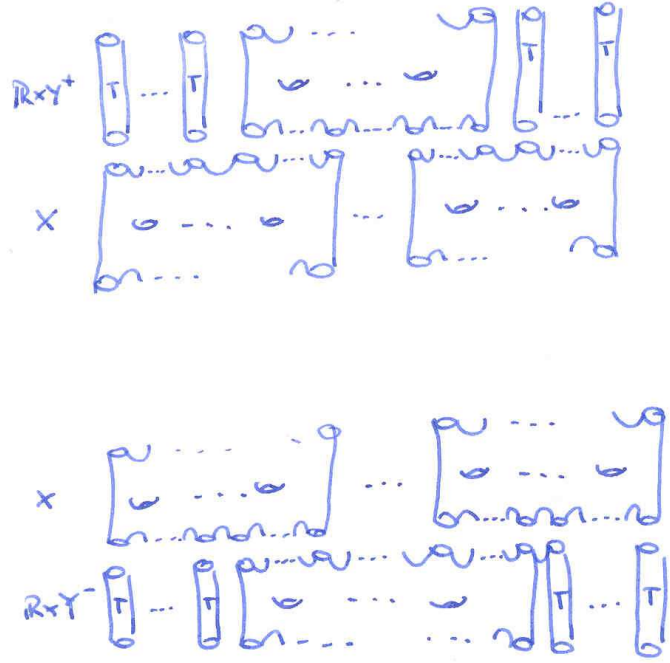}
\caption{Breaking of curves in a cobordism}
\label{fig:2}
\end{center}
\end{figure}
\begin{thm}[\cite{EGH}]\label{thm:master-F}
The potential $\F\in\frac{1}{\hbar}\DD$ satisfies the {\em master
  equation}  
\begin{equation}
\label{eq:master-F}
   e^\F\overleftarrow{\H^+}-\overrightarrow{\H^-}e^\F = 0. 
\end{equation} 
\end{thm}

{\bf Higher dimensional moduli spaces. }
So far we have considered only rigid holomorphic curves. The
generalization to higher-dimensional moduli spaces is straightforward
as follows, see~\cite{El} for details. We begin again with a contact manifold
$(Y,\lambda)$. A punctured $J$-holomorphic curve $f:\dot S\to\R\times
Y$ asymptotic to $\Gamma^\pm$ induces a continuous map $\bar f:\bar
S\to Y$ from the compactification $\bar S$ of $\dot S$ (replacing each
puncture by a boundary circle) to $Y$ with boundary on
$\Gamma^\pm$. Via the assignment $f\mapsto\bar f$ we view a moduli
space $\MM_g(\Gamma^-,\Gamma^+)/\R$ of dimension $d$ as a $d$-chain
$c_g(\Gamma^-,\Gamma^+)$ in the space of continuous maps with boundary
on $\Gamma^\pm$. Define correlators
$$
\langle
\underbrace{q,\dots,q}_{s^-};\underbrace{p,\dots,p}_{s^+}\rangle^Y_g
:= \sum_{|\Gamma^\pm|=s^\pm} c_g(\Gamma^-,\Gamma^+)
\,q^{\Gamma^-}p^{\Gamma^+}, 
$$
and the Hamiltonian encoding moduli spaces of all dimensions
$$
   \mathbb{H} := \frac{1}{\hbar}\sum_{g,s^-,s^+} \frac{1}{s^-!s^+!} \langle
   \underbrace{q,\dots,q}_{s^-};\underbrace{p,\dots,p}_{s^+}\rangle^Y_g
   \hbar^g.
$$
Then $\mathbb{H}\in\frac{1}{\hbar}\mathbb{W}$, where $\mathbb{W}$
consists of power series in $\hbar,p$, polynomial in $q$, with
coefficients being chains in suitable spaces of continuous maps to $Y$. 
$\mathbb{W}$ is equipped with the associative product $*$ combining
the $*$ product on $\WW$ with the obvious gluing of continuous
maps. Moreover, the boundary operator $\p$ on chains acts on
$\mathbb{W}$. Now Figure~\ref{fig:1} shows that $\mathbb{H}$ satisfies
the master equation 
\begin{equation}\label{eq:master-H2}
   \p\mathbb{H} + \frac{1}{2}\mathbb{H}*\mathbb{H} = 0.
\end{equation}
%
%

Similarly, for a symplectic cobordism $(X,\om)$ denote by $\mathbb{F}$ 
the potential encoding moduli spaces of all
dimensions. Figure~\ref{fig:2} shows that $\mathbb{F}$ satisfies the
master equation 
\begin{equation}
\label{eq:master-F2}
   \p e^{\mathbb{F}} = e^{\mathbb{F}}\overleftarrow{\mathbb{H}^+} -
   \overrightarrow{\mathbb{H}^-}e^{\mathbb{F}}. 
\end{equation} 
%
%

{\small 
\subsubsection*{Technical remarks}

The proofs of Theorems~\ref{thm:master-H} and~\ref{thm:master-F}
require a substantial amount of technical work which is still not
completed. The necessary compactness results are proven in
\cite{BEHWZ} (see also \cite{CM-comp}), and coherent orientations for the
moduli spaces are discussed in \cite{BM}. What remains to be proven
are the transversality and gluing results needed to describe the
structure of the moduli spaces. The best one can expect is that
$\MM_g(\Gamma^+,\Gamma^-)$ will be weighted branched manifolds in the
sense of \cite{CMS} with boundary and corners. The precise formulation
of the structure theorems is the subject of ongoing work of Hofer,
Wysocki and Zehnder~\cite{HWZ}, who introduce the machinery of
polyfolds as a tool to deal with the occuring analytic
difficulties. 
}

\section{The master equation with Lagrangian boundary
  conditions}\label{sec:master} 

In this section we show how degenerations of moduli spaces of punctured
holomorphic curves with Lagrangian boundary conditions lead to a
master equation similar to equation~\eqref{eq:master-F2}. The
discussion in this section is rather informal, the definitions being
postponed to Section~\ref{sec:string}. 

As in the previous section, let $(X^{2n},\om)$ be an exact symplectic
cobordism between contact manifolds $(Y^\pm,\lambda^\pm)$ with a
compatible almost complex structure $J$. In addition, let $Q\subset X$
be an exact compact oriented Lagrangian submanifold as described at
the beginning of Section~\ref{sec:SFT}. 

Fix ordered collections of closed Reeb orbits 
$\Gamma^\pm=(\gamma_1^\pm,\dots,\gamma_{s^\pm}^\pm)$ and integers
$g,k\geq 0$. Denote by
$$
   \MM_{g,k}(X,Q;\Gamma^-,\Gamma^+)
$$
the moduli space of connected $J$-holomorphic curves of genus $g$ in
$X$ with $s^+$ positive and $s^-$ negative punctures asymptotic to the
$\gamma_i^+$ resp.~$\gamma_j^-$ and with $k$ ordered boundary
components on $Q$. In view of the exactness hypotheses, this moduli
space has expected dimension   
\begin{eqnarray*}
\lefteqn{\dim\MM_{g,k}(X,Q;\Gamma^-,\Gamma^+)}\\
&=& (n-3)(2-2g-s^+-s^--k) +
   \sum_i\CZ(\gamma_i^+) - \sum_j\CZ(\gamma_j^-). 
\end{eqnarray*}
Simultaneous evaluation at the $k$ boundary circles defines a map
$$
c_{g,k}(\Gamma^-,\Gamma^+):\MM_{g,k}(X,Q;\Gamma^-,\Gamma^+)\to
\underbrace{\Sigma\x \dots \x\Sigma}_{k} =: \Sigma^k
$$ 
to the $k$-fold product of the space $\Sigma=\Sigma Q= C^0(S^1,Q)/S^1$
of closed strings ($=$ loops up to rotation in the domain) on $Q$. We
may thus view $c_{g,k}(\Gamma^-,\Gamma^+)$ as a chain in $\Sigma^k$. 
While this chain is the obvious geometric object to work with, for
various technical reasons it is more convenient to consider it as a
chain {\em relative to the constant strings}. Indeed, denote by $Q
\subset \Sigma Q$ the image of the embedding that assigns to each
point the constant string at that point, and set
$$
\const_k := \bigcup_{i=1}^k \Sigma \x \dots \x \stackrel{i}{Q} \x
\dots \x \Sigma \subset \Sigma^k.
$$
In these notes we view $c_{g,k}(\Gamma^-,\Gamma^+)$ as an element of
$C_*(\Sigma^k,\const_k)$. 

Define the correlators
$$
   \langle \underbrace{q,\dots,q}_{s^-};
   \underbrace{p,\dots,p}_{s^+}\rangle^{X,Q}_{g,k} 
   := \sum_{|\Gamma^\pm|=s^\pm}c_{g,k}(\Gamma^-,\Gamma^+)
   \,q^{\Gamma^-}p^{\Gamma^+},
$$
where the sum is taken over all ordered collections
$\Gamma^\pm$. Define the {\em potential} $\L$ of the pair $(X,Q)$ by  
$$
   \L:=\frac{1}{\hbar}\sum_{g=0}^\infty\L_g\hbar^g,
$$
where
$$
   \L_g := \sum_{s^-,s^+,k}\frac{1}{s^-!s^+!k!} \langle
   \underbrace{q,\dots,q}_{s^-};
   \underbrace{p,\dots,p}_{s^+}\rangle^{X,Q}_{g,k}.   
$$
Thus $\L$ is a formal power series in $\hbar$ and the
$p_\gamma^+$, polynomial in the $q_\gamma^-$, with coefficients in the
singular chains on $\Sigma^k$ (where $k$ ranges over all nonnegative
integers). Note that restricting to the terms of $\L$ corresponding to
$k=0$ we recover the usual potential $\F$ of the cobordism as
discussed in Section~\ref{sec:SFT}.

The following theorem encodes the codimension 1 boundaries of moduli
spaces (of any dimension) of punctured holomorphic curves with
boundary on $Q$. We state it here informally; the precise definition
of the operations $\Delta$ and $\nabla$ will be given in
Section~\ref{sec:string}. 

\begin{thm}\label{thm:master-L}
The potential $\L$ satisfies the {\em master equation}
\begin{equation}
\label{eq:master-L}
   (\p + \Delta + \hbar\nabla)e^\L =
   e^\L\overleftarrow{\H^+}-\overrightarrow{\H^-}e^\L, 
\end{equation} 
where $\p$ is the singular boundary operator and $\Delta$
resp.~$\nabla$ denote the operations of decomposing one string at a
self-intersection resp.~gluing two strings at an intersection.  
\end{thm}

\begin{proof}[Sketch of proof]
Figure~\ref{fig:3} shows the three new phenomena in the codimension one
boundary of a moduli space appearing in $\L$ (where for simplicity we
omitted negative punctures) which are due to the Lagrangian boundary
conditions.  
\begin{figure}[h]
\begin{center}
\epsfbox{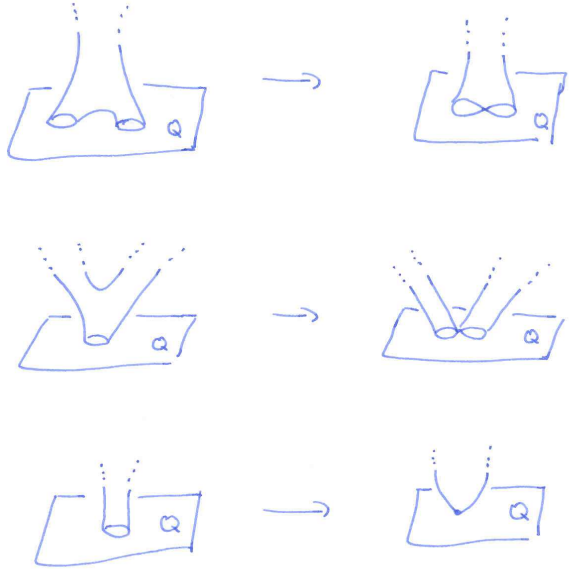}
\caption{Codimension 1 phenomena near the boundary}
\label{fig:3}
\end{center}
\end{figure}
%
The first one is the
pinching of an arc connecting two boundary loops. The inverse
procedure is decomposition of a boundary loop at a
self-intersection, which is encoded by the operation $\Delta$. 
The second boundary phenomenon is the pinching of an arc connecting
one boundary loop to itself. Here the inverse procedure is gluing two
boundary loops at an intersection, which is encoded by the operation
$\nabla$. The third one is the shrinking of a boundary loop to a
point, which leads to a chain with image in  $const_k \subset
\Sigma^k$. This part of the boundary is set to zero by working with
relative chains.

We point out that, as an abstract manifold, a moduli space in $\L$ has
other codimension 1 boundary components, e.g.~breaking off of higher
dimensional moduli spaces in a symplectization $\R\times
Y^\pm$. However, under evaluation at the boundary loops all these
moduli spaces lead to degenerate chains in $\Sigma^k$, i.e.~chains
that factor through chains of lower dimension, and therefore do not
appear in the master equation.  
\end{proof}

{\small 
\subsubsection*{Technical remarks}

Theorem~\ref{thm:master-L} is conjectural at this point, as it
requires analytic results even beyond those needed for the discussion
of SFT. However, there are clear strategies for attacking the basic
issues.

The discussion of coherent orientations should reduce to a careful
combination of the results in ~\cite{BM} for the SFT case and
\cite{FOOO} for the case of disks without punctures. There is one new
phenomenon which has not been described in the literature yet, which
deals with ``self-gluing'' at a transversal self-intersection of a
boundary curve as in the operation $\Delta$. We expect to treat
orientations completely in \cite{CL:exact}.

The necessary compactness results essentially reduce to a
combination of standard SFT compactness~\cite{BEHWZ} and the
observation that our exactness assumption rules out bubbling of
holomorphic spheres or disks. Then the remaining phenomena can be
studied in terms of degenerations of domains, using a standard
doubling trick. There is however one new feature that at
the moment we do not know how to deal with properly: In dimension
$n>3$ index considerations yield an upper bound on the number of
boundary components on any curve with given set of asymptotics
prescribed at the punctures and representing a fixed relative
homology class. We do not expect such a result for $n=2,3$. 

To achieve transversality for moduli spaces of punctured curves
with boundaries, we hope to adapt the techniques in \cite{CM-trans}.
Note that the word transversality is used in two senses here, as it
means both that moduli spaces are ``transversally cut out'' and that
the resulting chains on string space have certain transversality
properties needed to apply the string topology operations
(cf. Section~\ref{sec:string}). The gluing results will require a
combination of SFT-type gluing with boundary gluing as in \cite{FOOO}.
}

\section{String topology}\label{sec:string}

In this section we discuss the new operations on the chains of string
space introduced by Chas and Sullivan~\cite{CS,CS2}.

Let $Q$ be a closed oriented smooth $n$-manifold, and define the {\em
free loop space} $\Lambda Q$ as
$$
\Lambda := \Lambda Q := C^0(S^1,Q).
$$
In the previous section, we also defined the {\em closed string space}
$\Sigma Q$ of $Q$ by   
$$
   \Sigma:=\Sigma    Q:=\Lambda Q/S^1. 
$$
For $k\geq 1$ let
$\Sigma^k:=\Sigma\times\dots\times\Sigma$ be the $k$-fold product.  

A {\em smooth $i$-chain}
in $\Sigma^k$ is a tuple
$(K;\sigma_1,\dots,\sigma_k)$ where $K$ is some $i$-dimensional
compact connected oriented manifold with corners,
and for each $1\leq i \leq k$ the $\sigma_i:P_i \to Q$ are smooth maps
from the total space of some circle bundle $p_i:P_i \to K$ into
$Q$. To get a good {\em set} of chains, one actually needs to require
that the domains, together with all the circle bundles, come with
fixed embeddings into some $\R^N$. Denote by $\wt C_i(\Sigma^k)$ the
$\Q$-vector space with basis the smooth $i$-chains in $\Sigma^k$, and
by $C_i(\Sigma^k)$ the quotient space where we factor out the
degenerate $i$-chains and identify $(-K,\sigma_1,\dots,\sigma_k)$ with
$-(K,\sigma_1,\dots,\sigma_k)$, where $-K$ denotes $K$ with the
opposite orientation. 

Chas and Sullivan introduced partially defined coproduct and product
operations 
\begin{gather*}
   \delta: C_i(\Sigma) \to C_{i+2-n}(\Sigma \x \Sigma), \cr
   \mu: C_i(\Sigma \x \Sigma) \to C_{i+2-n}(\Sigma)
\end{gather*}
as follows. 
To define the coproduct, let a smooth $i$-chain $(K;\sigma)\in
C_i(\Sigma)$ be given and consider the pullback bundle $p^*(P) \to P$. 
Note that this bundle has a tautological section, so that the
induced map $p^*P \to Q$ can be viewed as a map $\wt\sigma:P \to
\Lambda Q$. Denote by $K_{\delta\sigma}\subset P \x [0,1]$ the
closure of the inverse 
image of the diagonal of $Q \x Q$ under the evaluation map  
\begin{align*}
ev_\sigma: P \x  (0,1) &\to Q \x Q\\
(p,t) &\mapsto (\wt{\sigma}(p)(0),\wt{\sigma}(p)(t)).
\end{align*}
Assuming enough transversality, $K_{\delta\sigma}$ is a submanifold with
corners of $P \x [0,1]$ of dimension $i +2-n$. 
Now the chain $\delta\sigma: K_{\delta\sigma}
\to \Sigma \x \Sigma$ is defined by assigning to $(p,t) \in
K_{\delta\sigma}$ the ordered pair of strings corresponding to the
pair of loops $(\wt{\sigma}(p)_{|[0,t]},\wt{\sigma}(p)_{|[t,1]})$.   
Note that the two corresponding circle bundles are trivial, since the
point where we cut the old circle gives rise to a section.

Similarly, to define the product, we start from a smooth $i$-chain
$(K;\sigma_1,\sigma_2)\in C_i(\Sigma^2)$ and consider the fiber
product $\pi:P \to K$ of the two circle bundles. Then again the
pull-backs $\pi^*(P_i) \to P$ have tautological sections, yielding
two families of loops $\wt{\sigma}_i:P \to \Lambda Q$, $i=1,2$. 
Then we consider the evaluation map
\begin{align*}
ev: P &\to Q \x Q\\ 
p &\mapsto (\wt{\sigma}_1(p)(0),\wt{\sigma}_2(p)(0)).
\end{align*}
Again assuming transversality, $K_{\mu\sigma}$ is a submanifold of
$P$ of dimension $i+2-n$, and the chain
$\mu\sigma:K_{\mu\sigma} \to \Sigma$ is 
defined by assigning to $p \in K_{\mu\sigma}$ the string
corresponding to the concatenation of the two loops
$\wt{\sigma}_1(p)$ and $\wt{\sigma}_2(p)$. The circle bundle
involved is again trivial.

More generally, on $C_*(\Sigma^k)$ we can define operations $\delta_r$
and $\mu_{r_1,r_2}$ where the above operations are applied to the
corresponding strings. Note that $\delta$ does not commute with the
boundary operator, because of the possible ``vanishing of small
loops'' (cf. Figure~\ref{fig:4} ). 
\begin{figure}[h]
\begin{center}
\epsfbox{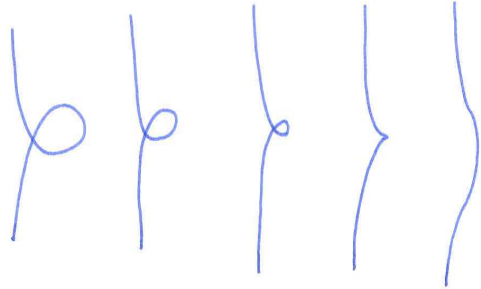}
\caption{A small loop disappears}
\label{fig:4}
\end{center}
\end{figure}
However, after passing to the chains relative
to the constant strings $C_*(\Sigma^k,\const_k)$ as in the previous section, 
we find that both operations behave well with respect to the boundary
operator. Note that an application of $\mu$ to a pair of strings of
which one is constant leads to a degenerate chain, so the operations
indeed descend to the relative chains with our conventions.

Motivated by the grading conventions of SFT, we
now consider the shifted chain complex
$$
\hat \chains_{i,k}:= C_{i+k(3-n)}(\Sigma^k, \const_k), 
$$
and note that in this new grading the operations $\delta_r$ and
$\mu_{r_1,r_2}$ have bidegrees $(-1,1)$ and $(-2(n-3)-1,-1)$,
respectively. We let the symmetric group on $k$ letters act on $\hat
\chains_{i,k}$ by 
$$
\rho \cdot \left(t \mapsto \left(\sigma_1(t),\dots \sigma_k(t)\right)\right)
:= (\sgn \rho)^{3-n} \left(t \mapsto \left(\sigma_{\rho(1)}(t),\dots
  \sigma_{\rho(k)}(t)\right)\right),
$$
and denote by $\chains_{i,k}$ the quotient of $\hat \chains_{i,k}$ by
this action. 
Then the direct sum 
$$
\chains := \Q \oplus \bigoplus_{k\geq 1, i \geq k(n-3)}
\chains_{i,k} 
$$
forms an algebra with unit 
$\one \in \Q$ under the multiplication $\cdot$ defined as follows.  
Given chains $\sigma \in \chains_{i,k}$ and $\tau \in \chains_{j,l}$ we
consider the new chain $\sigma \x \tau \in \chains_{i+j,k+l}$ which is defined
on the product of the domains and is given at the parameter value $(t,s)$ by
the concatenation of the collections of strings $\sigma(t)$ and
$\tau(s)$. Then we set 
$$
\sigma \cdot \tau := (-1)^{il(3-n)}\sigma \x \tau.
$$
One can check that
$$
\sigma \cdot \tau = (-1)^{ij}\tau \cdot \sigma,
$$
i.e. the multiplication is graded commutative. 
Denote by $\p$ the usual boundary operator on $\chains_{*,k}$, shifted
by $(-1)^{k(3-n)}$. Then a short calculation shows that $\p$ is a
derivation of the algebra $\chains$ with respect to the grading by 
the shifted dimension of chains.   
We now define an operation $\Delta: \chains \to \chains$ of bidegree
$(-1,1)$ by $\Delta(\one)=0$ and
\begin{equation}\label{eq:def_delta}
\Delta_{|\hat\chains_{*,k}} :=
\sum\limits_{r=1}^k (-1)^{(r+k)(3-n)} \delta_r  
\end{equation}
for $k \geq 1$, and an operation $\nabla: \chains \to \chains$ of bidegree
$(-2(n-3)-1,-1)$ by $\nabla(\one)=0$ and
\begin{eqnarray}
\nabla_{|\hat\chains_{*,k}} &:=& \left\{ 
\begin{array}{cl} 
0 & \text{ if } k=1\\
\sum\limits_{r_1<r_2} (-1)^{(r_2-1+k)(3-n)} \mu_{r_1,r_2} & \text{ if
} k>1
\end{array} \right.
\label{eq:def_nabla}
\end{eqnarray}
These are the operations appearing in the statement of
Theorem~\ref{thm:master-L}. One checks that they are indeed
well-defined on $\chains$. The following proposition is implicit in
\cite{CS2} and will be proven in \cite{CL:string}.
\begin{prop}\label{prop:string}
On sufficiently transverse chains one has the following identities:
\begin{enumerate}
\item $\p \Delta + \Delta \p = 0$,
\item $\Delta^2 = 0$,
\item $\Delta(c_1\cdot c_2) = \Delta(c_1)\cdot c_2 +
  (-1)^{|c_1|}c_1\cdot \Delta(c_2)$,
\item $\p \nabla + \nabla \p = 0$,
\item $\nabla^2=0$,
\item \begin{eqnarray}
\lefteqn{\nabla(c_1c_2c_3)= }\notag\\
& \nabla(c_1c_2)c_3
+(-1)^{|c_1|}c_1\nabla(c_2c_3) 
+(-1)^{|c_2||c_3|}\nabla(c_1c_3)c_2\notag\\
&  - \nabla(c_1)c_2c_3 
-(-1)^{|c_1|}c_1\nabla(c_2)c_3 
-(-1)^{|c_1|+|c_2|} c_1c_2\nabla(c_3).
\label{eq:7terms}
\end{eqnarray}
\item $\Delta \nabla + \nabla \Delta = 0$.
\end{enumerate}
\end{prop}

Denote by
$$
\HH_k:=\HH_k(\Sigma,Q) := H_{k+3-n}(\Sigma,Q)
$$
the shifted homology of string space relative to the constant strings.
Part (i) and (iv) of the proposition imply that $\delta$ and $\mu$
give rise to operations
\begin{gather*}
\delta: \HH_* \to \HH_* \otimes \HH_* \\
\mu: \HH_* \otimes \HH_* \to \HH_*
\end{gather*}
of degrees $-1$ and $-2(n-3)-1$, respectively. The other properties
listed in the proposition imply the following theorem of Chas and
Sullivan~\cite{CS2}. 
\begin{thm}[Chas-Sullivan~\cite{CS2}]\label{thm:string}
The operations $\mu$ and $\delta$ give the string homology $\HH_*(\Sigma,Q)$
of any oriented manifold $Q$ the structure of an involutive Lie bialgebra 
of bidegree $(-1,-2(n-3)-1)$.  
\end{thm}
The precise definition of this concept will be given in the next
section.

{\small 
\subsubsection*{Technical remarks}

The statement of Proposition~\ref{prop:string} should be understood as
follows: Each of the operations $\Delta$, $\nabla$, $\Delta^2$, $\nabla^2$,
$\Delta \nabla$ and $\nabla \Delta$ has a domain of definition
described in terms of certain jet transversality conditions for
evaluation maps, and the stated identities hold on the intersections
of the domains for the operations involved. This, as well as
orientation issues, will be discussed in detail in \cite{CL:string}.

Note that in order to give the statement of Theorem~\ref{thm:master-L}
rigorous meaning, one also needs to make sense of ``diagonal terms''
like $\nabla(\L \cdot \L)$. These involve bracket-type operations for
a chain with itself, which usually does not satisfy the transversality
requirements. There are several solutions to this problem, see
e.g. \cite{Fu:06} for one which works in the context of holomorphic
curves. 
}

\section{Some algebra}\label{sec:alg}

In this section, we describe the abstract algebraic structures
underlying both SFT and string topology. 
Throughout, we fix an integer $n$ and use the symbol
$\hbar$ to denote a formal variable of degree $2(n-3)$. 
We also fix a field $\K$ of characteristic zero. 

\begin{definition}
A {\em (strictly commutative) BV$_\infty$-algebra} is a pair $(A,D)$,
where $A$ is a graded commutative algebra over $\K$ with unit and $D:
A[[\hbar]] \to A[[\hbar]]$ is a linear map of degree $-1$ satisfying
\begin{enumerate}[({BV}1)]
\item $DD=0$,
\item $D$ admits an expansion
\begin{equation}\label{eq:D-decomp}
D= \frac 1 \hbar \sum_{k=1}^\infty D^k \hbar^k,
\end{equation}
where each $D^k:A \to A$ is a differential operator of order $\leq
k$ on $A$, and
\item D(1)=0.
\end{enumerate}
\end{definition}
For the convenience of the reader, we recall the inductive definition
of a differential operator over a morphism $m:A \to B$ of graded
commutative algebras with unit: The zero map from $A$ to $B$ is a 
differential operator of order $-1$, and a linear map $D:A \to B$
is a differential operator of order $\leq k$ over $m$ if for each
$a\in A$ the map $x \mapsto D(xa) - D(x)m(a)$ is a
differential operator of order $\leq k-1$. A differential operator on $A$
is by definition a differential operator over the identity $\one:A \to A$.
For a polynomial algebra this reduces to the usual concept, where a
differential operator of order $\leq k$ is one taking at most $k$ partial
derivatives.

Since all the BV$_\infty$-algebras appearing in these notes will be
strictly commutative, we will permit ourselves to drop this qualifier
from now on. Note that the requirement $DD=0$ can be written as a
sequence of quadratic identities satisfied by the $D^k$ in the
decomposition \eqref{eq:D-decomp}. The first few of these read
$$
D^1D^1=0, \quad D^1D^2 + D^2D^1=0, \quad D^1D^3 + D^2D^2 + D^3D^1=0,
\quad \text{\rm etc.}
$$
In particular, $D^1$ is a derivation of $A$ and $D^2$ descends to the
homology algebra $H(A,D^1)$ as a differential operator of order $\leq 2$ and
square 0, so that $(H(A,D^1),D^2)$ is a BV-algebra (cf. \cite{Ge:1}). 
The deviation of
$D^2$ from being a derivation is measured by the so-called BV-bracket
$$
[a,b]_D := (-1)^{|a|}(D(ab)- D(a)b - (-1)^{|a|}aD(b)),
$$
which is an odd Poisson bracket on $H(A,D^1)$.

Next, we need the notion of a morphism of 
BV$_\infty$-algebras. We will only give an ad hoc definition in the
special case that the domain of the morphism is a free commutative
algebra $A= S(V)$, where $S(V)$ denotes the symmetric tensor product
of the graded vector space $V$. Given a linear map $\pHi:S(V) \to
B[[\hbar]]$, we define a map $e^{\pHi}:S(V)[[\hbar]] \to B[[\hbar]]$ as
\begin{eqnarray*}
\lefteqn{e^\pHi(v_1\cdots v_k):=}\\
& & \sum_{\ell=1}^k\frac{1}{\ell!}
   \hspace{-3mm}
   \sum_{i_1+\dots+i_\ell=k}\frac{1}{i_1!\cdots i_\ell!}
   \sum_{\rho\in S_k}\eps(\rho) \pHi(v_{\rho(1)}\cdots
   v_{\rho(i_1)})\cdots \pHi(v_{\rho(k-i_\ell+1)}\cdots
   v_{\rho(k)}), 
\end{eqnarray*}
where $\eps(\rho)$ is the sign one gets from rearranging the $v_j$
according to the permutation $\rho$, and the map is extended to
$S(V)[[\hbar]]$ by $\K[[\hbar]]$-linearity. 

\comment{
Note that $S(V)$ is both an algebra and a coalgebra, where the
comultiplication $c:S(V) \to S(V) \otimes S(V)$ is given as 
$$
c(v_1\cdots v_p) := \sum_{i=0}^p \sum_{\rho \in S_p \atop {\rho(1)<
    \dots \rho(i) \atop \rho(i+1) < \dots \rho(p)}} \eps(\rho)
v_{\rho(1)}\cdots v_{\rho(i)} \otimes v_{\rho(i+1)}\cdots v_{\rho(p)}.
$$
Here we agree that empty products equal $1\in S^0(V)$, and
$\eps(\rho)$ is the sign one gets from rearranging the $v_j$ according
to the permutation $\rho$. One can check that $c$ is in fact
cocommutative. 
More generally, one can define 
$c^{(r)}: S(V) \to S(V)^{\otimes r+1}$ inductively as
$c^{(1)}=c$ and $c^{(r+1)} := (\one_{S(V)} \otimes c) c^{(r)}$ for
$r\geq 1$. For any linear map $\pHi:S(V)[[\hbar]] \to B[[\hbar]]$ of
algebras satisfying $\pHi(1)=0$ one can define a map
$e^{\pHi}:S(V)[[\hbar]] \to B[[\hbar]]$ as
$$
e^\pHi:= \sum_{k=0}^\infty \frac 1 {k!} \pHi^{\otimes k} \circ
c^{(k-1)},
$$
where for $k\geq 1$  the map $f^{\otimes k}:S(V)^{\otimes k} \to B$
maps $a_1 \otimes \dots \otimes a_k$ to $f(a_1)\cdots
f(a_k)$. Moreover we agree that $c^{(-1)}$ is the projection $S(V) \to
S^0(V) \equiv \K$ and $f^{\otimes 0}: \K \to \K$ is the identity.
}

\begin{definition}
A {\em BV$_\infty$-morphism} from a free BV$_\infty$-algebra $(A=S(V),
D_A)$ to a BV$_\infty$-algebra $(B,D_B)$ is a linear map $\pHi: A \to
B[[\hbar]]$ of degree $0$ satisfying
\begin{enumerate}[({BVmor}1)]
\item $\pHi(1)=0$,
\item $e^{\pHi}D_A = D_B e^{\pHi}$, and
\item $\pHi$ admits an expansion
\begin{equation}\label{eq:pHi-decomp}
\pHi = \frac 1 \hbar \sum_{k=1}^\infty \pHi^k \hbar^k,
\end{equation}
where each $\pHi^k:A \to B$ is a differential operator of order $\leq k$
over the zero morphism, i.e. $\pHi^k(v_1\cdots v_l)=0$ for $l>k$.
\end{enumerate}
An {\em augmentation} of a free BV$_\infty$-algebra is
a morphism from $(A=S(V),D_A)$ to the trivial BV$_\infty$-algebra $(\K,0)$.
\end{definition}
As before for $D$, the equation in (BVmor2) can be expanded in powers
of $\hbar$. The leading order term will have the form
$$
e^{\pHi^1}D_A^1 = D_B^1 e^{\pHi^1}.
$$
Note that since $\pHi^1$ can only be non-zero on $V \cong S^1(V) \subset
S(V)$, the associated map $e^{\pHi^1}:S(V) \to B$ is an
algebra morphism. The above equation then says that it is in fact a
morphism of differential graded algebras from $(A,D_A^1)$ to
$(B,D_B^1)$, and so it induces a morphism between the corresponding
homology algebras. We also remark that if $\beta:A[[\hbar]] \to
\K[[\hbar]]$ is an augmentation of the BV$_\infty$-algebra $(A,D_A)$,
then $e^{\beta^1}$ is an augmentation of the differential graded
algebra $(A,D_A^1)$. 

Now let $(A=S(V),D)$ be a free BV$_\infty$-algebra and suppose we are
given an augmentation $\beta:S(V) \to \K[[\hbar]]$. From it we
construct a linear map $\Phi=\Phi^\beta:A \to A$ as
$$
\Phi(v_1\cdots v_k) = v_1\cdots v_k + \sum_{l=1}^k \frac 1 {l!(k-l)!}
\sum_{\rho \in S_k} \eps(\rho) e^{\beta}(v_{\rho(1)}\cdots v_{\rho(l)}) 
   v_{\rho(l+1)}\cdots v_{\rho(k)}.
$$
One can prove by induction that there is a unique map $f:S(V) \to
S(V)[[\hbar]]$ such that $\Phi=e^f$, and $f$ satisfies
properties (BVmor1) and (BVmor3) of a BV$_\infty$-morphism. One also
proves that $\Phi$ is invertible. 
Moreover, by construction we have $\pi_0 \Phi =
e^{\beta}$, where $\pi_0:S(V) \to V^0\cong \K$ is the projection to
the constants.  Now one can introduce a {\em twisted
BV$_\infty$-operator} $D^\beta$ on $A$ by setting
$$
D^\beta := \Phi D \Phi^{-1}.
$$
This new BV$_\infty$-operator on $A$ has the additional property that
{\em it has no constant terms}, in the sense that $\pi_0 D^\beta(a)=0$ for
any $a\in A[[\hbar]]$. This is because
$$
\pi_0 D^\beta (a) = \pi_0 \Phi D \Phi^{-1}(a) = e^{\beta} D
(\Phi^{-1}(a)) = 0
$$
since $\beta$ was an augmentation. It is now clear by construction that
$\Phi$ is an isomorphism between the BV$_\infty$-algebras $(A,D)$ and
$(A,D^\beta)$. 

BV$_\infty$-operators without
constant terms on $S(V)$ are really nice objects, as they can be 
linearized to give an ``infinity involutive Lie bialgebra structure''
on $V$. We will not do this here, but in the remainder of this section we
explain in detail the construction of the linearized homology and 
outline the construction of the Lie bialgebra structure it carries. 

Expanding $\Phi=\sum_{k=0}^\infty \Phi^k \hbar^k$ in powers of
$\hbar$, the leading order term is a map $\Phi^0:A \to A$ which has the form
\begin{eqnarray*}
\lefteqn{\Phi^0(v_1\cdots v_k) }\\
&=& v_1\cdots v_k + \sum_{l=1}^k \frac 1 {l!(k-l)!}
\sum_{\rho \in S_k} \eps(\rho) e^{\beta^1}(v_{\rho(1)}\cdots v_{\rho(l)}) 
   v_{\rho(l+1)}\cdots v_{\rho(k)}\\
&=& v_1 \cdots v_k + \sum_{l=1}^k \frac 1 {l!(k-l)!}
\sum_{\rho \in S_k} \eps(\rho) \beta^1(v_{\rho(1)})\cdots \beta^1(v_{\rho(l)}) 
   v_{\rho(l+1)}\cdots v_{\rho(k)}\\
&=& (v_1+ \beta^1(v_1)) \cdots (v_k + \beta^1(v_k)),
\end{eqnarray*}
so it is an algebra automomorphism. As above for $\Phi$, one can check
that the constant term $\pi_0 \Phi^0$ equals $e^{\beta^1}:A \to \K$,
which is an augmentation of the differential graded algebra $(A,D^1)$. 
It is now standard~\cite{Ch} to consider the {\em twisted derivation}
$$
\p^\beta := \Phi^0 D^1 (\Phi^0)^{-1},
$$
which obviously still squares to 0. One can check that in fact
$\p^\beta$ is nothing but the leading term $(D^\beta)^1$ in the
expansion of $D^\beta$. Being a derivation, it is
completely determined by the restriction $\p^\beta|_{V \subset
S(V)}$, which can be expanded as
$$
\p^\beta|_V = \p^\beta_0 + \p^\beta_1 + \p^\beta_2 + \dots
$$
with $\p^\beta_k:V \to S^k(V)$. Now note that $\p^\beta_0=
e^{\beta^1}D^1 (\Phi_0)^{-1} = 0$.  
It follows that $\p^\lin:=\p^\beta_1:V \to V$ defines a
differential. We call its homology {\em the linearized homology of the
augmented BV$_\infty$-algebra} $(A,D,\beta)$ and denote it by
$H^\lin(A,D,\beta)$.   

One can view the higher order terms $\p^\beta_k:V \to S^k(V)$ in the
expansion as operations which (after introducing suitable signs)
combine into an odd co-L$_\infty$-structure on $V$. In particular,
$\p^\beta_2:V \to S^2(V)$ gives rise to a map 
$$
\delta:V \to V \otimes V, \qquad v \mapsto (\iota \otimes
\one)c (\p^\beta_2(v)),
$$
where $\iota(w)=(-1)^{|w|}w$ and $c:S^2V \to V \otimes V$ is given on
generators as $c(v_1v_2)= v_1 \otimes v_2 + (-1)^{|v_1||v_2|}v_2
\otimes v_1$. One can check that $\delta$ descends to 
the linearized homology and gives it the structure
of a co-Lie algebra of degree $-1$, in the sense of the following 
\begin{definition}
A {\em co-Lie algebra of degree $d$} is a graded vector space $V$ with a
linear map $\delta:V \to V \otimes V$ of degree $d$ such that
\begin{enumerate}
\item 
$$
\tau_d \delta = -\delta,
$$ 
where $\tau_d:V^{\otimes 2} \to V^{\otimes 2}$ is given on decomposable
elements as  
$$
\tau_d(v_1 \otimes v_2):= (-1)^{(|v_1|+d)(|v_2|+d)} v_2 \otimes v_1,
$$
and
\item 
$$
0=(1 + \rho_d + \rho_d^2)( \delta \otimes \one)\delta:V \to
  V^{\otimes 3},
$$ 
where $\rho_d:V^{\otimes 3} \to V^{\otimes 3}$ is given on
decomposable elements as  
$$
\rho_d(v_1 \otimes v_2 \otimes v_3) = (-1)^{(|v_1|+|v_2|)(|v_3|+d)}v_3
\otimes v_1 \otimes v_2.
$$
\end{enumerate}
\end{definition}
To get a Lie bracket on the linearized homology, we expand the twisted
BV$_\infty$-operator $D^\beta$ as 
$$
D^\beta = \frac 1 \hbar \sum_k D^{\beta,k}\hbar^k
$$ 
and further consider the component $D^{\beta,2}_1:S^2(V) \to V$ of $D^{\beta,2}$.
Then one can check that the map
$$
\mu:V \otimes V \to V, \qquad \mu(v_1,v_2) =
(-1)^{|v_1|}D^{\beta,2}_1(v_1v_2)
$$
descends to the linearized homology and gives it the structure of a
graded Lie algebra of degree $-2(n-3)-1$, in the following sense.
\begin{definition}
A {\em Lie algebra of degree $d$} is a graded vector space $V$ together with
a linear map $\mu:V \otimes V \to V$ of degree $d$ such that
\begin{enumerate}
\item $\mu \tau_d = -\mu$, and
\item $0=\mu(\mu \otimes \one) (1 + \rho_d + \rho_d^2):V^{\otimes 3}
  \to V$.
\end{enumerate}
Here $\tau_d$ and $\rho_d$ are as in the previous definition.
\end{definition}
In fact, our situation is even better, since the compatibility
conditions in the BV$_\infty$-algebra force a strong link between the
Lie algebra and co-Lie algebra structures on linearized homology. 
\begin{definition}
A {\em Lie bialgebra of bidegree $(d_1,d_2)$} is a graded vector
space $V$ with linear maps $\delta:V \to V \otimes V$ of degree $d_1$
and $\mu:V \otimes V \to V$ of degree $d_2$ such that
\begin{enumerate}
\item $(V,\delta)$ is a co-Lie algebra of degree $d_1$,
\item $(V,\mu)$ is a Lie algebra of degree $d_2$, and
\item $\delta$ and $\mu$ satisfy Drinfeld compatibility, i.e.
\begin{align*}
\delta \mu(a,b) =& \sum a_1 \otimes \mu(a_2,b) + \mu(a,b_1) \otimes b_2\\
&+ \sum (-1)^{|a||b|+|a|+|b|} 
\left(\mu(b,a_1)\otimes a_2 + b_1 \otimes \mu(b_2,a)\right), 
\end{align*}
where $\delta(a)=\sum a_1 \otimes a_2$ and $\delta(b)=\sum b_1 \otimes b_2$.
\end{enumerate}
The Lie bialgebra is called {\em involutive} if moreover $0=\mu \delta:V
\to V$.
\end{definition}
Note that if $\delta$ and $\mu$ have different parity, then
involutivity is automatically satisfied on the part of the Lie
bialgebra with grading $\equiv d_1 \mod 2$.  

Our discussion on augmentations can now be summarized in the following
claim. 
\begin{prop}\label{prop:LieBi}
Any augmentation $\beta$ of a free BV$_\infty$-algebra $(A=S(V),D)$
gives rise to a linearized homology $H^\lin(A,D,\beta)$ which has the
structure of an involutive Lie bialgebra of bidegree $(-1,-2(n-3)-1)$.
\end{prop}
A morphism between augmented BV$_\infty$-algebras $(A,D_A,\alpha)$ and
$(B,D_B,\beta)$ is a BV$_\infty$-morphism $\pHi$ which in addition
satisfies
$$
e^\alpha = e^\beta \circ e^\pHi.
$$
The following statement will be proven in \cite{CFL}.
\begin{prop}\label{prop:LieBi-mor}
Any morphism between augmented BV$_\infty$-algebras $(A,D_A,\alpha)$
and $(B,D_B,\beta)$ gives rise to a chain map for the linearized
complexes which induces a morphism of involutive Lie bialgebras in
homology. 
\end{prop}
\bigskip

The following concepts will become important in Section~\ref{sec:cob}.
\begin{definition}
A {\em Maurer-Cartan element} in a BV$_\infty$-algebra $(A,D)$ is an
element $a\in \frac 1 \hbar A[[\hbar]]$ of degree zero satisfying
$$
D(e^a)=0.
$$
\end{definition}
Given a Maurer-Cartan element $a\in \frac 1 \hbar A[[\hbar]]$, one
defines the {\em twisted BV$_\infty$-operator} 
$$
D_a:A[[\hbar]] \to A[[\hbar]], \qquad b \mapsto e^{-a}D(e^ab).
$$
One can check that $D_a$ indeed maps between the spaces given and
$(A,D_a)$ is again a BV$_\infty$-algebra. In the special case that 
the original BV-operator $D$ had the property that $D^k=0$ for all
$k\geq 3$, the condition $D(e^a)=0$ simplifies to
\begin{equation}\label{eq:MC}
Da + \frac \hbar 2 [a,a]_D = 0.
\end{equation}
\comment{
In this case, $D_a$ can be written as
$$
D_a(b)= Db + \hbar [a,b]_D.
$$
In particular, $D_a-D$ is a derivation, so that $D_a$ generates the
same bracket as $D$.
}
{\small 
\subsubsection*{Technical remarks}

The proper concept of BV$_\infty$-algebra is more general than the one
presented here; it is discussed by Tamarkin and Tsygan~\cite{TT}. What
we consider here is essentially equivalent to what
Kravchenko~\cite{Kr} calls homotopy BV-algebras.  

In the augmented free case it is more appropriate to discuss
everything in terms of involutive Lie bialgebras up to homotopy, since
in this case our BV$_\infty$-algebras are just cobar constructions on
the co-Lie part of the structure. The relevant algebra will be
developed in~\cite{CFL}. 


However, the object that arises naturally from the SFT of a contact
manifold without filling is only a weak version of such involutive
BL$_\infty$ algebras, meaning that in general one does have $n$-to-$0$
operations. We do not know how to properly treat them on the linear
level, i.e. without going to the cobar construction.

One essential aspect which is completely missing here is the discussion
of homotopies. These are necessary to properly formulate the invarince
properties of the various algebraic objects constructed in both SFT
and string topology, see~\cite{CFL}.}

\section{Back to geometry}\label{sec:geo}

In this section, we take another look at SFT and string topology and
describe them using the language developed in the previous section.

We have seen in Section~\ref{sec:SFT} that to a closed contact
manifold $(Y^{2n-1},\lambda)$ and a suitable choice of almost complex
structure on $\R \x Y$ one can associate an element $\H \in \frac 1
\hbar \WW$, where $\WW$ is the graded Weyl algebra constructed
from the (good) closed Reeb orbits of $\lambda$. In particular, $\WW$
contains the graded commutative subalgebra $\AA$ of polynomials in the
variables $q_\gamma$. 

We claim that the differential operator $\D=\D_\SFT:=
\overrightarrow{\H}: \AA[[\hbar]] \to \AA[[\hbar]]$ defined as in
Section~\ref{sec:SFT} via the replacements
$$
   p_\gamma\mapsto \kappa_\gamma \hbar 
       \overrightarrow{\frac{\p}{\p q_\gamma}}
$$
gives $\AA$ the structure of a BV$_\infty$-algebra as defined in the
previous section. Indeed, the star product of the Weyl algebra exactly
models the behavior of our differential operators under composition,
so that the condition $\D\D=0$ is a direct translation of $\H * \H=0$. 
Next note that each differentiation contributes one power of $\hbar$,
so that the coefficient of $\hbar^{k-1}$ in the expansion of $\D$ is a
differential operator of order $\leq k$. Finally the condition $\D(1)=0$
follows from the fact that there are no curves without positive
punctures in $\R\times Y$, so every term of $\H$ contains at
least one $p$ and $\D$ does not have any zero order terms as a
differential operator. 

Tracing through the definitions, one finds that the holomorphic curves
which correspond to the term $\D^k$ in the expansion of $\D$ are
precisely those for which the number of positive punctures and the
genus add up to $k$. In particular, we find that $\D^1$ corresponds to
curves of genus zero with exactly one positive puncture, so that
$H_*(\AA,\D_\SFT^1)$ is just the usual contact homology of $Y$. The
BV-operator $\D^2$ consists of two terms: a differential operator of order two
corresponding to genus zero curves with two positive punctures and a
derivation corresponding to curves of genus one with one positive
puncture. Note that, in particular, the corresponding BV-bracket is
determined just by the genus zero curves.

Next we discuss the morphism arising from an exact symplectic
cobordism $(X,\om)$ from $(Y^+,\lambda^+)$ to
$(Y^-,\lambda^-)$. Indeed, the potential $\F$ of the cobordism $X$ is
a power series in the variables $p^+_\gamma$ and $\hbar$ with coefficients
polynomial in the $q^-_\gamma$. Analogously to the above discussion
for $\H$, it gives rise to a map $\overrightarrow{\F}: \AA^+ 
\to \AA^-[[\hbar]]$. We claim that this map is a BV$_\infty$-morphism
as defined in the previous section. 

Again each term of $\F$ contains at least one
$p^+$, so that we have $\overrightarrow{\F}(1)=0$. Simple algebraic
manipulations show that the master equation of
Theorem~\ref{thm:master-F} is equivalent to the fact that
$$
e^{\overrightarrow \F} \D^+ = \D^- e^{\overrightarrow \F}.
$$
The statement about the order of the differential operators in the
expansion of $\overrightarrow{\F}$ with respect to $\hbar$ follows as
in the case of $\D$. Also, as for $\D$, the curves contributing to
$\F^k$ are precisely those for which the number of positive punctures
and the genus add up to $k$. In particular, the chain map yielding a
morphism between the contact homologies is given by curves of genus
zero with exactly one positive puncture.

Augmentations arise in a geometric way by considering symplectic
cobordisms with only a positive end, since the BV$_\infty$-algebra
associated to the empty contact manifold is $(\K,0)$. From
Proposition~\ref{prop:LieBi} we then obtain 
\begin{thm}\label{thm:filling_linear}
Every exact symplectic filling $(X,d\lambda)$ of a closed contact
manifold $(Y^{2n-1},\lambda)$ gives rise to a linearized contact
homology $HC^\lin(X,\lambda)$ which has the structure of an involutive
Lie bialgebra of bidegree $(-1,-2(n-3)-1)$.
\end{thm}
To get a feeling for this invariant of the symplectic filling, it may be
instructive to describe the operations involved in the definition of
linearized contact homology in terms of the broken holomorphic curves
they count. 

We start by looking at the boundary operator $\p^\lin$ of linearized
contact homology. We claim that it counts configurations of curves as
depicted in Figure~\ref{fig:5}. 
\begin{figure}[ht]
\begin{center}
\epsfbox{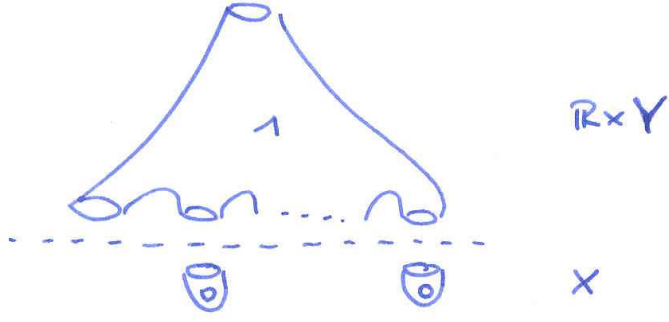}
\caption{The linearized boundary operator}
\label{fig:5}
\end{center}
\end{figure}
Here the numbers written in the various non-trivial components denote
the index of the corresponding family of curves. To understand this
picture, note first that the augmentation is given by the potential
$\F$ of the filling. In particular, the augmentation
$e^{\overrightarrow{\F^1}}:(\AA, \D^1_\SFT) \to \Q$ of the underlying
differential graded algebra is determined by the count of finite
energy planes in $X$ of index 0. The algebra automorphism $\Phi^0:\AA
\to \AA$ is given on generators as $\Phi^0(q_\gamma) = q_\gamma +
\overrightarrow{\F^1}(q_\gamma)$. Recall that $\p^\lin= \Phi^0 \D^1
(\Phi^0)^{-1}$, so that
$$
\p^{\overrightarrow{\F}}(q_\gamma) = 
  \Phi^0 \D^1 (q_\gamma - \overrightarrow{\F^1}(q_\gamma)) = \Phi^0
  \D^1(q_\gamma). 
$$ 
Now one can see that the linear part $\p^\lin$ corresponds precisely
to configurations as in Figure~\ref{fig:5}. 
%
%
We remark in passing that
the constant term in $\p^{\overrightarrow{\F}}(q_\gamma)$ counts exactly the
rigid configurations of the type pictured in Figure~\ref{fig:6}. 
\begin{figure}[h]
\begin{center}
\epsfbox{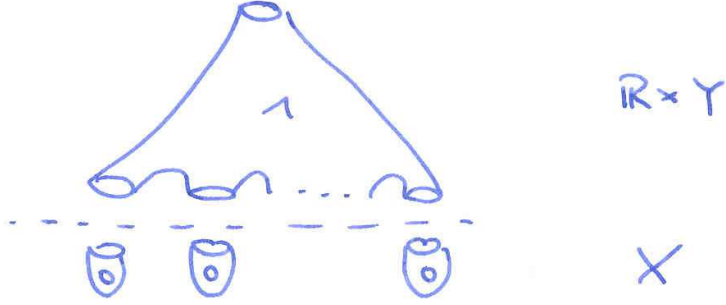}
\caption{The constant term}
\label{fig:6}
\end{center}
\end{figure}

The fact that it is zero corresponds to the geometric fact that these
configurations together form the oriented boundary of the
corresponding moduli space of finite energy planes in $X$ of index 1
asymptotic to $\gamma$. 

The same reasoning as for the linearized boundary operator $\p^\lin$
shows that the coefficients of the cobracket on linearized homology
are given by the counts of configurations as depicted in
Figure~\ref{fig:7}. 
\begin{figure}[h]
\begin{center}
\epsfbox{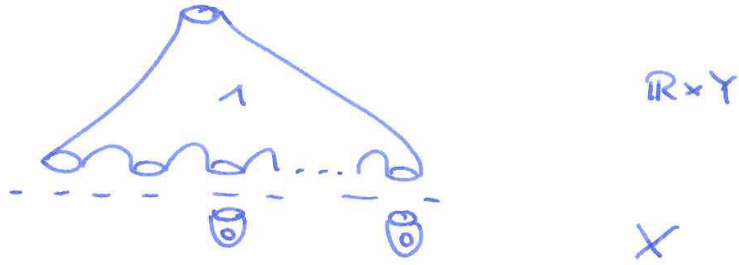}
\caption{The cobracket on linearized homology}
\label{fig:7}
\end{center}
\end{figure}

Something more interesting happens for the bracket,
where we get the two types of pictures seen in Figure~\ref{fig:8}. 
\begin{figure}[h]
\begin{center}
\epsfbox{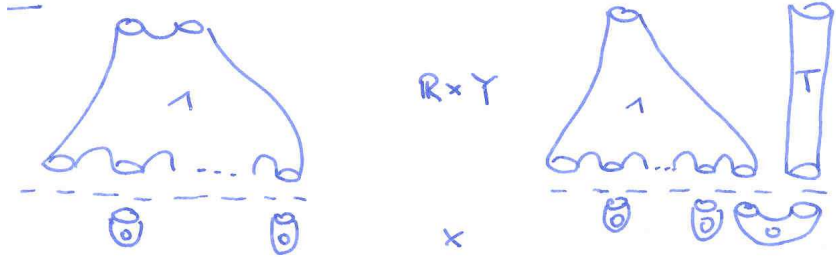}
\caption{The bracket on linearized homology}
\label{fig:8}
\end{center}
\end{figure}

\bigskip

We next turn to string topology. Here things are somewhat more
complicated, as the operations $\Delta$ and $\nabla$ are only
partially defined. Still the compatibility conditions listed in
Proposition~\ref{prop:string} suggest that $\D_\str:= \p + \Delta +
\hbar \nabla: \chains[[\hbar]] \to \chains[[\hbar]]$ is reasonably
viewed as a (partially defined) BV$_\infty$-operator on $\chains$. 
Note that in our grading conventions it is indeed homogeneous of
degree $-1$.

Now assume $(X,d\lambda)$ is an exact filling of the contact manifold
$(Y,\lambda)$ and $Q \subset X$ is an exact Lagrangian. In this
situation the master equation of Theorem~\ref{thm:master-L} simplifies
to
$$
\D_\str e^\L = e^\L \overleftarrow{\H^+}.
$$
So replacing $\F$ by $\L$ in the discussion before, we find that
$\overrightarrow{\L}:\AA \to \CC[[\hbar]]$ defines a
BV$_\infty$-morphism between the BV$_\infty$-algebra $(\AA,\D_\SFT)$
of the contact manifold $(Y,\lambda)$ and the (partial) BV$_\infty$-algebra 
$(\CC, \D_\str)$ constructed from the string topology of $Q$.  

One can actually do better. Split $\L=\F + \L'$, where $\F$ is the
potential of the filling $(X,d\lambda)$ of $Y$, i.e. the part of $\L$
corresponding to curves with no boundary. It gives an
augmentation $\overrightarrow{\F}$ for $(\AA,\D_\SFT)$, leading 
to the twisted BV$_\infty$-operator 
$$
\D^\F_\SFT := \overrightarrow{e^{\F}} \D_\SFT \overrightarrow{e^{-\F}}.
$$
on $\AA$. As discussed before, this twisted BV$_\infty$-operator does
not have constant terms, so it admits a linearization. 
Moreover, denoting by $\AA^1\subset \AA$ the linear
polynomials in the $q$'s, i.e. the vector space generated by the good
closed Reeb orbits, we obtain
\begin{thm}\label{thm:filling_mor}
The part $\L^1_1$ of $\L'$ corresponding to curves with one puncture
and one boundary component gives rise to a chain map from $\AA^1$ to
$\CC^1=C_*(\Sigma Q, Q)$ which induces a morphism of involutive
graded Lie bialgebras
$$
HC^{\lin}(X,\lambda) \to H(\Sigma Q, Q).
$$
\end{thm}
There should be a more conceptual form of Theorem~\ref{thm:filling_mor}
giving a morphism of the chain-level structures.

{\small 
\subsubsection*{Technical remarks}

The proof of Theorem~\ref{thm:filling_linear} uses only the algebraic
properties of SFT described in \cite{EGH}. While
Theorem~\ref{thm:filling_mor} is a similar consequence of
Theorem~\ref{thm:master-L}, in \cite{CL:exact} we will provide a
complete, more direct proof which is independent of the master equation.
}

\section{Cotangent bundles}\label{sec:cot}

Here we discuss the important special case of cotangent bundles.
More precisely, let $X= T^*Q$ with its canonical symplectic form and
$Q \subset X$ be the zero section. In this context,
Theorem~\ref{thm:filling_mor} yields a morphism of involutive graded Lie
bialgebras 
$$
\psi:HC^\lin(T^*Q, \lambda) \to H(\Sigma Q, Q).
$$
The main result of this section is
\begin{thm}\label{thm:iso}
The map $\psi:HC^\lin(T^*Q, \lambda) \to H(\Sigma Q, Q)$ is an
isomorphism of involutive graded Lie bialgebras. 
\end{thm}
The proof of this theorem is based on an action/length estimate for
holomorphic curves. For this, we pick a Riemannian metric on $Q$ and
let $Y=S^*Q$ be the unit cotangent bundle for this metric with the
canonical 1-form $\lambda=p\,dq$. Then closed Reeb orbits $\gamma$ on
$(Y,\lambda)$ are lifts of closed geodesics on $Q$ (again denoted by
$\gamma$) and their action $\int_\gamma\lambda$ equals the length
$\length(\gamma)$ of the geodesic. 

We pick a compatible almost complex structure $J_\rho$ on $T^*Q$ of
the form 
$$
   J_\rho:\frac{\p}{\p q_i}\mapsto -\rho(|p|)\frac{\p}{\p p_i},\qquad
   \rho(|p|)\frac{\p}{\p p_i}\mapsto\frac{\p}{\p q_i},
$$
where $q_i$ are geodesic normal coordinates on $Q$ and $p_i$ are dual
coordinates. Here $\rho:[0,\infty)\to(0,\infty)$ is a non-decreasing
function with $\rho(r)\equiv 1$ near $r=0$ and $\rho(r)=r$ for $r$
large. For this almost complex structure, every $T$-periodic geodesic
$\gamma$ gives rise to two natural $J_\rho$-holomorphic curves: the
{\em cylinder over $\gamma$}
$$
   \R\times\R/\Z\to T^*Q,\qquad
   (s,t)\mapsto\Bigl(\gamma(Tt),g(s)\dot\gamma(Tt)\Bigr), 
$$
where $g:\R\to\R$ is the solution of $g'(s)=T\rho(s)$ with $g(0)=0$,
and its restriction to $[0,\infty)\times\R/\Z$ which we call the {\em
half-cylinder over $\gamma$}. 

\begin{lemma}\label{lem:action-estimate}
Let $J_\rho$ be as above. Then for any $J_\rho$-holomorphic curve $f$
in $T^*Q$ with finitely many positive punctures asymptotic to closed
geodesics $\gamma_1,\dots,\gamma_s$ and with finitely many boundary
loops $\sigma_1,\dots,\sigma_k$ on $Q$ we have 
$$
\sum_{i=1}^k\length(\sigma_i) \leq \sum_{j=1}^s\length(\gamma_j).
$$
Equality holds if and only if $f$ is a branched cover of the
half-cylinder over a closed geodesic.  
\end{lemma}

The idea of the proof is the following. Consider the 1-form 
$$
   \lambda := \frac{p\,dq}{|p|}
$$
on $T^*Q \setminus Q$. A straightforward computation shows that
$f^*d\lambda\geq 0$ for every $J_\rho$-holomorphic curve
$f$. Moreover, the integral of $f^*\lambda$ over a small loop around
the $j$-th positive puncture converges to $\length(\gamma_i)$ as the
loop shrinks to a point (this is immediate), and the integral of
$f^*\lambda$ over a loop around the $i$-th boundary component  
converges to $\length(\sigma_i)$ as the loop converges to the
boundary (this is a bit tricky). Now the inequality follows from
Stokes' theorem: 
$$
   0\leq \int_fd\lambda = \sum_{j=1}^s\length(\gamma_j) -
   \sum_{i=1}^k\length(\sigma_i). 
$$
Moreover, this argument shows that equality holds iff
$f^*d\lambda\equiv 0$, which happens precisely for branched covers of
orbit half-cylinders. 
To make this idea rigorous, however, some care has to be taken when $f$
hits $Q$ at interior points since $f^*d\lambda$ is not defined
there. 

Given Lemma~\ref{lem:action-estimate}, the proof of
Theorem~\ref{thm:iso} follows a standard argument, using the
filtration by actions on the SFT side and by length on the string
topology side and applying the Five Lemma . 

There is, however, one important point. In the chain complex for
linearized contact homology, only good Reeb orbits appear. One checks
that in the present context a Reeb orbit is bad if and only if the
stable bundle of the corresponding $S^1$-family of geodesics in $Q$,
viewed as a critical manifold of the energy functional, is
non-orientable. In this case, the local Morse homology
vanishes~\cite{Ra}, so that it does not contribute to the homology on
the string topology side either.  

\subsubsection*{An Example: Surfaces of genus $>2$}

To illustrate Theorem~\ref{thm:iso}, we discuss the special case where
$Q$ is a closed oriented surface of genus $g\geq 2$. In this case
the space of strings in each non-trivial free homotopy class is
contractible, so that with our grading conventions the string homology
is concentrated in degree $-1$ and equals
$$
\HH_*(\Sigma, Q) =\HH_{-1}(\Sigma Q, Q) = \bigoplus_{0 \neq \gamma \in
  [S^1,Q]} \Q \langle \gamma \rangle.
$$
The bracket and cobracket on $\HH_*$ in this case were discovered by
Goldman~\cite{Go} and Turaev~\cite{Tu}, respectively, and were an
important precursor for the definition of string topology by Chas and
Sullivan. Note that here the abstract definitions given in
Section~\ref{sec:string} reduce to the following simple prescription:
For the bracket, represent the two given free homotopy classes by curves
in general position, so that they intersect transversally in finitely many 
points. Now each intersection point contributes one summand to the
bracket. It is simply the free homotopy class obtained by
concatenation of the two loops based there, with a sign obtained by
comparing the ordered set of tangent vectors with the
orientation of the surface, see Figure~\ref{fig:9}. 
\begin{figure}[h]
\begin{center}
\epsfbox{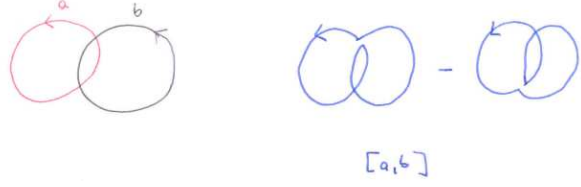}
\caption{The bracket for curves on a surface}
\label{fig:9}
\end{center}
\end{figure}

Similarly, for the cobracket,
represent a given free homotopy class by a curve in general position
with itself. Then the cobracket has two summands for each self-intersection
point. They are given by the two ordered pairs of the classes obtained by
splitting the closed curve into two, again each with the sign obtained by
comparing the ordered set of tangent vectors to the surface
orientation, see Figure~\ref{fig:10}. Here the color coding is such
that the red string is the first and the black string is the second.
Note, in particular, that the bracket and cobracket can be effectively
computed for any given input.
\begin{figure}[h]
\begin{center}
\epsfbox{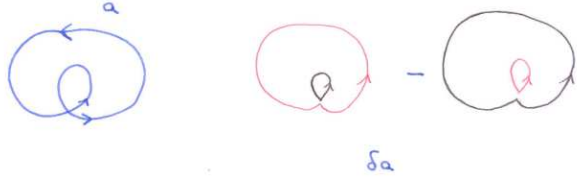}
\caption{The cobracket for curves on a surface}
\label{fig:10}
\end{center}
\end{figure}
%
%
%
%

Our next goal is to translate this topological information into
existence results for holomorphic curves in $\R \x S^*Q$. For this we
choose some Riemannian 
metric of negative curvature on our surface $Q$. Then all closed Reeb
orbits are non-contractible and their Morse indices, and hence their
Conley-Zehnder indices, are zero. Note that the dimension formulas for
moduli spaces of curves in the symplectization $\R \x S^*Q$ and in the
cotangent bundle $T^*Q$ discussed in Section~\ref{sec:SFT} and
\ref{sec:master} specialize to
\begin{gather*}
\dim \MM_g(\Gamma^-,\Gamma^+) = -(2-2g -s^+-s^-)\\
\dim \MM_{g,k}(T^*Q,Q; \Gamma^+) = - (2-2g -s^+-k).
\end{gather*}
Since the potential $\F$ is built from $0$-dimensional
moduli spaces of curves in $T^*Q$, we find that only genus zero curves
with exactly two (positive) punctures contribute, so that
$$
\F = \frac 1 \hbar \sum_{i,j} e^{ij}p_ip_j.
$$
Using a maximum principle argument, one can prove that the
only two-punctur\-ed spheres in $T^*Q$ are the trivial cylinders
connecting $\gamma$ with $-\gamma$, so that $e^{ij}=\pm 1$ if
the indices $i$ and $j$ correspond to one geodesic with its two possible
orientations and $e^{ij}=0$ in all other cases. So, up to sign, we
have completely determined the coefficients of $\F$.

Similarly, since $\H$ is built from $1$-dimensional moduli spaces of
curves in the symplectization, we find that there are only four
possible types of curves that contribute. In fact, $(g,s^+,s^-)$ can
only take the values 
$$
(0,1,2), (0,2,1), (0,3,0) \text{ \rm and } (1,1,0),
$$
so that $\H$ has the form
$$
\H = \frac 1 \hbar \left(
     \sum_{i,j,k} a^k_{ij} q_iq_jp_k + \sum_{i,j,k} b^{jk}_i q_ip_jp_k
    +\sum_{i,j,k} c^{ijk} p_ip_jp_k + \hbar \sum_i d^i p_i\right).
$$
Now recall that the chain map inducing the isomorphism of
Theorem~\ref{thm:iso} is given by moduli spaces of holomorphic disks
with one puncture and boundary on $Q \subset T^*Q$. Since the
geodesics are absolutely length-minimizing in their homotopy classes,
Lemma~\ref{lem:action-estimate} implies that the only such
half-cylinders are the trivial half-cylinders over the geodesics
themselves. So the isomorphism is given by the map that sends each
closed geodesic, viewed as a Reeb orbit, to its free homotopy class.
In view of Theorem~\ref{thm:iso}, and combined with the explicit
knowledge of $\F$, this allows one to compute the coefficients
$a_{ij}^k$, $b_i^{jk}$ and $c^{ijk}$ in $\H$ algebraically. 
For better illustration, here we instead present the pictures behind
the algebra, which also allow us to determine the coefficients $d^i$.

To understand the counts $a_{ij}^k$ of spheres in the symplectization
with 1 positive and 2 negative punctures, we consider the boundary of a
$1$-dimensional moduli space of spheres in $T^*Q$ with two boundary
\begin{figure}[h]
\begin{center}
\epsfbox{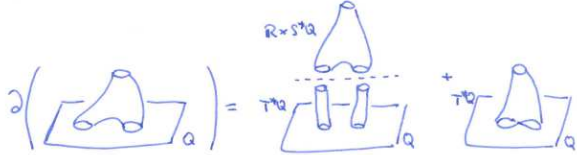}
\caption{Codimension 1 phenomena in a moduli space of curves with two
  boundary components and one puncture}
\label{fig:11}
\end{center}
\end{figure}
components on $Q$ and one puncture, given in Figure~\ref{fig:11}. 
We find
that the $a_{ij}^k$ are (up to sign) identical to the coefficients of
the string cobracket (represented by the second picture). 

To understand the count $b_i^{jk}$ of spheres in the symplectization
with 2 positive and 1 negative punctures, we consider the boundary of
a $1$-dimensional moduli space of spheres in $T^*Q$ with 1 boundary
\begin{figure}[h]
\begin{center}
\epsfbox{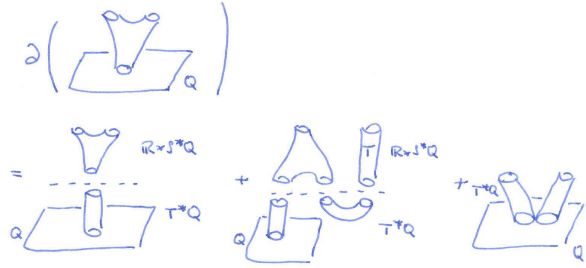}
\caption{Codimension 1 phenomena in a moduli space of curves with one
  boundary component and two punctures}
\label{fig:12}
\end{center}
\end{figure}
component on $Q$ and 2 punctures, given in Figure~\ref{fig:12}. 
Again this gives a linear equation, expressing the coefficients $b_i^{jk}$
in terms of the coefficients of the string bracket and the string
cobracket. 

To understand the count $c^{ijk}$ of spheres in the symplectization
with 3 positive punctures, we consider the boundary of a
$1$-dimensional moduli space of spheres in $T^*Q$ with 3 punctures
(and no boundary), given in Figure~\ref{fig:13}, and conclude that the
$c^{ijk}$ depend linearly on the $a_{ij}^k$ and the $b_i^{jk}$. 
\begin{figure}[h]
\begin{center}
\epsfbox{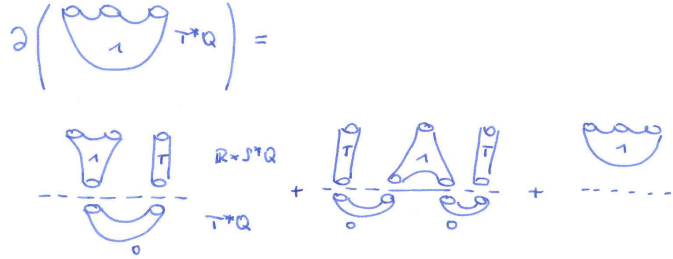}
\caption{Codimension 1 phenomena in a moduli space of curves with
  three punctures and no boundary}
\label{fig:13}
\end{center}
\end{figure}

Finally, to understand the count $d^i$ of tori in the symplectization
with 1 positive puncture, we consider the boundary of a
$1$-dimensional moduli space of tori in $T^*Q$ with 1 puncture (and no
boundary), pictured in Figure~\ref{fig:14}, and conclude that the count
is determined by the number of times a geodesic appears with both its
orientations in the expression for the cobracket. 
\begin{figure}[h]
\begin{center}
\epsfbox{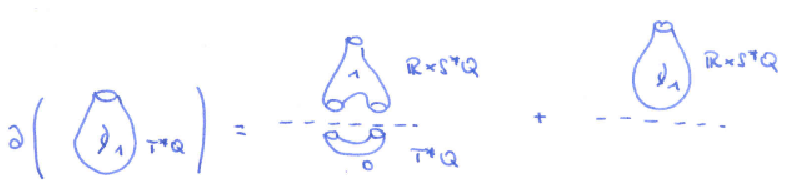}
\caption{Codimension 1 phenomena in a moduli space of curves with
  genus 1, one puncture and no boundary}
\label{fig:14}
\end{center}
\end{figure}

Taken together, we have sketched the proof of
\begin{prop}\label{prop:surfaces}
On an oriented closed surface $Q$ of genus $\geq 2$ with a metric of
negative curvature, the coefficients of $\H$ are linear functions 
of the coefficients of the string bracket and cobracket. In
particular, they are independent of the choice of the metric of
negative curvature.
\end{prop}
Note that each type of coefficient is non-zero for an appropriate
choice of asymptotic geodesics. Specific examples are given in
Figures~\ref{fig:15}--\ref{fig:17}. 
\begin{figure}[h]
\begin{center}
\epsfbox{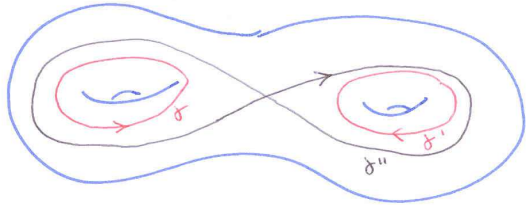}
\caption{Nonvanishing string cobracket}
\label{fig:15}
\end{center}
\end{figure}
Here Figure~\ref{fig:15} shows a
configuration where the coefficient of $\frac 1 \hbar
q_{\gamma}q_{\gamma'}p_{\gamma''}$ is nonzero, and Figure~\ref{fig:16}
shows a configuration where the coefficient of $\frac 1 \hbar
q_{\gamma}p_{\gamma'}p_{\gamma''}$ is nonzero. 
\begin{figure}[h]
\begin{center}
\epsfbox{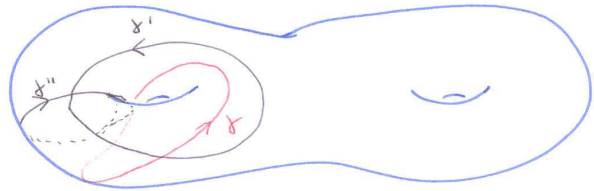}
\caption{Nonvanishing string bracket}
\label{fig:16}
\end{center}
\end{figure}
In both cases, all other possible combinations of string brackets and
cobrackets either vanish (Figure~\ref{fig:15}) or involve other
homotopy classes (Figure~\ref{fig:16}), so that they
also give examples for a  nonzero coefficient of $\frac 1 \hbar
p_{-\gamma}p_{-\gamma'}p_{\gamma''}$ (Figure~\ref{fig:15}) and $\frac 1 \hbar
p_{-\gamma}p_{\gamma'}p_{\gamma''}$ (Figure~\ref{fig:16}), respectively. 

Finally, Figure~\ref{fig:17} gives an example for a nonzero
\begin{figure}[h]
\begin{center}
\epsfbox{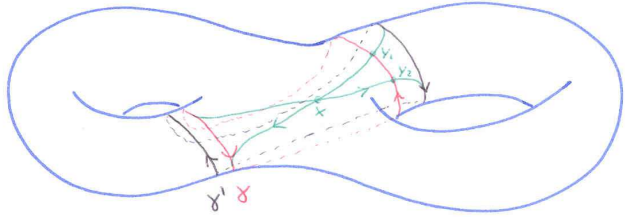}
\caption{A curve appears with both orientations in a cobracket}
\label{fig:17}
\end{center}
\end{figure}
coefficient of $p_\gamma$. Start with two curves $\gamma$ and
$\gamma'$ which differ only in their orientation (we have drawn them as
parallel copies to improve readability of the picture). Then build a
new curve $\gamma''$ by joining $\gamma$ and $\gamma'$ as shown in
green. The curve $\gamma''$ has three self-intersections, labelled by
$x$, $y_1$ and $y_2$ in the picture. Applying the cobracket to
$\gamma''$ means resolving each of them in turn. The resolution of $x$
clearly gives back $\gamma$ and $\gamma'$, whereas one can check that
the resolution of $y_1$ or $y_2$ gives curves in other homotopy
classes. 

{\small 
\subsubsection*{Technical remarks}

The details of the proof of Theorem~\ref{thm:iso}, of course including
the proof of Lemma~\ref{lem:action-estimate}, will be given in
\cite{CL:exact}. There we also give the precise form of the equations
alluded to in Proposition~\ref{prop:surfaces}, which depends on a
detailed discussion of orientation conventions.
}

\section{General symplectic cobordisms}\label{sec:cob}

In the previous sections we have considered only exact symplectic
cobordisms. In this section we outline the generalization to non-exact
cobordisms in terms of twisted $BV_\infty$-structures and discuss
relations to Fukaya's work on Lagrangian embeddings. 

In this section, let $(X^{2n},\om)$ be a symplectic
cobordism between contact manifolds $(Y^\pm,\lambda^\pm)$
(i.e.~$\om|_{Y^\pm}=d\lambda^\pm$) with a compatible almost complex
structure $J$. In addition, let $Q\subset X$ be a closed oriented
relatively spin Lagrangian submanifold. We allow both $(X,\om)$ and
$Q$ to be non-exact. For simplicity, we assume that $H_1(X;\Z)$
has no torsion. 

In the non-exact case we need to incorporate homology classes into the 
formalism, which is done as described in~\cite{EGH}: 
Fix a basis $A_1,\dots,A_N$ for $H_2(X,Q)$ (modulo torsion)
and introduce formal variables $z_1,\dots,z_N$ with grading
$$
|z_i| := -2c_1(A_i),
$$
where $c_1$ denotes the first Chern class of the tangent bundle
$(TX,J)$. Choose curves $b_1,\dots,b_M$ representing a basis of
$H_1(X;\Z)$ (which is free by assumption). For each closed Reeb orbit 
$\gamma$ of $Y^+$ and $Y^-$ we fix a 
capping surface connecting $\gamma$ to the corresponding linear
combination of the $b_i$. These
choices allow us to assign a class in $H_2(X,Q)$ to each
punctured holomorphic curve in $X$ with boundary on $Q$.  

Fix ordered collections of closed Reeb orbits 
$\Gamma^\pm=(\gamma_1^\pm,\dots,\gamma_{s^\pm}^\pm)$, integers
$g,k\geq 0$, and a relative homology class $d\in H_2(X,Q)$.  
Denote by
$$
   \MM^d_{g,k}(X,Q;\Gamma^-,\Gamma^+)
$$
the moduli space of $J$-holomorphic curves of genus $g$ in $X$ with
$s^+$ positive and $s^-$ negative punctures asymptotic to the
$\gamma_i^+$ resp.~$\gamma_j^-$ and with $k$ ordered boundary
components on $Q$ representing the relative homology class $d$. This
moduli space has expected dimension  
\begin{eqnarray*}
\lefteqn{\dim\MM^d_{g,k}(X,Q;\Gamma^-,\Gamma^+)}\\
&=& (n-3)(2-2g-s^+-s^--k) +
   \sum_i\CZ(\gamma_i^+) - \sum_j\CZ(\gamma_j^-) + \mu(d),  
\end{eqnarray*}
where $\mu:H_2(X,Q)\to\Z$ denotes the Maslov class. 
Note that $\mu$ takes only even values because $Q$ is orientable.
Simultaneous evaluation at the $k$ boundary components defines a map
$$
c^d_{g,k}(\Gamma^-,\Gamma^+):\MM^d_{g,k}(X,Q;\Gamma^-,\Gamma^+)\to
\underbrace{\Sigma\x \dots \x\Sigma}_{k} =: \Sigma^k,
$$ 
which we again view as an element of the chain group
$C_*(\Sigma^k,\const_k)$ defined in Section~\ref{sec:master}. 
As in the exact case, define the correlators
$$
\langle\underbrace{q,\dots,q}_{s^-};
\underbrace{p,\dots,p}_{s^+}\rangle^d_{g,k} 
:= \sum_{|\Gamma^\pm|=s^\pm}c^d_{g,k}(\Gamma^-,\Gamma^+)
\,q^{\Gamma^-}p^{\Gamma^+}
$$
and the potential $\L$ of the pair $(X,Q)$ by 
$$
   \L:=\frac{1}{\hbar}\sum_{g=0}^\infty\L_g\hbar^g,
$$
where
$$
   \L_g := \sum_{s^-,s^+,k,d}\frac 1 {s^-!s^+!k!} \langle
   \underbrace{q,\dots,q}_{s^-};
   \underbrace{p,\dots,p}_{s^+}\rangle^d_{g,k}\,z^d.   
$$
Thus $\L$ is a formal power series in $\hbar$, $z^{\pm 1}$ and the
$p_\gamma^+$, polynomial in the $q_\gamma^-$, with coefficients in the
string algebra $\CC$ of $Q$. Similarly, the Hamiltonians $\H^\pm$ of
$Y^\pm$ now include the variables $z$ encoding relative homology
classes obtained from the embeddings $Y^\pm \subset X$. With these
modified definitions, the proof of Theorem~\ref{thm:master-L} carries
over verbatim to show  

\begin{thm}\label{thm:master-L2}
The potential $\L$ of a (not necessarily exact) pair $(X,Q)$ satisfies
the master equation
\begin{equation}
\label{eq:master-L2}
   \D_\str(e^\L) =
   e^\L\overleftarrow{\H^+}-\overrightarrow{\H^-}e^\L, 
\end{equation} 
where $\D_\str = \p + \Delta + \hbar\nabla$ is the
$BV_\infty$-operator in string topology as defined in
Section~\ref{sec:geo}.  
\end{thm}

Formally, equation~\eqref{eq:master-L2} is identical to
equation~\eqref{eq:master-L}, but we will see that the interpretation
is quite different. For simplicity, we assume from now on
$Y^-=\varnothing$. Write
$$
   \L = \A+\B,
$$
where $\A := \L|_{p^+=0}$ encodes the holomorphic curves without
punctures. The master equation specializes at $p^+=0$ to
\begin{equation}\label{eq:master-A}
   \D_\str(e^\A)=0, 
\end{equation}
so $\A$ defines a Maurer-Cartan element in the $BV_\infty$-algebra
$(\CC[[z]],\D_\str)$. As explained in Section~\ref{sec:alg},
this leads to a twisted $BV_\infty$-operator 
$$
   \D_\A := e^{-\A}\D_\str e^\A
$$
on $\CC[[z]]$. Now arguing as in Section~\ref{sec:geo} we find
\begin{cor}\label{cor:mor2}
The map 
$$
   \PHI:\AA^+[[\hbar]]\to\CC[[\hbar,z]],\qquad
   g\mapsto e^{\overrightarrow \B}g
$$
is a $BV_\infty$-morphism from $(\AA^+[[z]],\D_\SFT)$ to
$(\CC[[z]],\D_\A)$. 
\end{cor}

  From here the theory in the general case is entirely parallel to the
one in the exact case. The only difference is that non-exact
symplectic cobordisms lead to $BV_\infty$-morphisms with respect to
{\em twisted} $BV_\infty$-structures on the Lagrangian/negative side.

To conclude this section, let us explain how equation~\eqref{eq:master-A}
is related to Fukaya's work~\cite{Fu:06}. First note that the
 BV$_\infty$-operator $\D_\str$ of string topology does not have
higher order terms, so that according to \eqref{eq:MC} 
equation \eqref{eq:master-A} is equivalent to
$$
\D_\str\A + \frac \hbar 2 [\A,\A]_\str = 0.
$$

Following a general principle known
from SFT, there is a version of equation~\eqref{eq:master-A} for
the part $\a:= \A_0$ of the generating function $\A$ corresponding to
genus zero curves, which is obtained by considering the leading order
term of the expansion of this equation in terms of $\hbar$. It takes
the form  
\begin{equation}   \label{eq:master-a} 
   (\p+\Delta)\a+\frac{1}{2}[\a,\a]_{\rm string}=0.
\end{equation}
\comment{
This equation can be interpreted as saying that $\a$ has zero
curvature when viewed as a connection on the differential Poisson
algebra $\bigl(\CC[[z]],\p+\Delta,[\ ,\
]_\str\bigr)$. 
This in turn implies that the covariant derivative 
$$
   \d_\a:\CC[[z]]\to\CC[[z]],\qquad
   \sigma\mapsto (\p+\Delta) \sigma + [\a,\sigma]_{\rm string}
$$ 
defines a differential on the algebra $\CC[[z]]$ whose 
homology $H_*(\CC[[z]],\d_\a)$ is an invariant of the
Lagrangian embedding up to Hamiltonian isotopy.
}
Specializing even further to the part $\a_1$ of $\a$ encoding genus 
zero curves with only one boundary component (i.e.~holomorphic disks),
we recover the equation 
$$
   \p\a_1+\frac{1}{2}[\a_1,\a_1]_{\rm string}=0,
$$
which was first observed by Fukaya~\cite{Fu:06}. In that paper, Fukaya
introduces the idea of viewing moduli spaces of holomorphic disks as
chains in string and loop space via evaluation at the boundary. He
also outlines numerous applications of this circle of ideas to the
topology of Lagrangian embeddings. These include the
characterization of all aspherical irreducible closed oriented
Lagrangian submanifolds in $\C^3$ (they must be a product of a circle
with a surface) and a proof of a generalization of Audin's conjecture
that each Lagrangian torus in $\C^n$ has minimal Maslov number two.  

We remark that each of the potentials $\A$, $\a$ and $\a_1$ give rise
to a twisted version of the corresponding differential in string
topology. The resulting homologies are invariants of the Lagrangian
embedding up to Hamiltonian isotopy.

We conclude this section with a discussion of the relation between
holomorphic curves with punctures and curves with boundary. Consider
again a closed Lagrangian submanifold $Q$ in a symplectic cobordism
$X$ without negative end. We have seen that the part $\A=\L|_{p^+=0}$
without positive punctures of the potential $\L$ yields a twisted
$BV_\infty$-operator $\D_\A$ on the algebra $\CC[[z]]$ of chains in
string space. On the other hand, cutting out $Q$ we get a symplectic
cobordism $X\setminus Q$ with negative end $S^*Q$. In this context,
the discussion following Theorem~\ref{thm:master-L2} shows that the   
part $\G=\F|_{p^+=0}$ without positive punctures of the potential
$\F$ of $X\setminus Q$ yields a twisted $BV_\infty$-operator
$\D_\G$ on the algebra $\AA[[z]]$ of polynomials in $q_\gamma$
corresponding to the Reeb orbits on $S^*Q$. The following result
states that these twisted operators are related by an 
isomorphism as in Theorem~\ref{thm:iso}.

\begin{thm}\label{thm:twisted-iso}
In the situation above, there exists a $BV_\infty$-morphism 
between the twisted $BV_\infty$-algebras
$$
   \Phi: (\AA[[z]],\D_\G)\to(\CC[[z]],\D_\A)
$$
which induces an isomorphism on homology. 
\end{thm}

Note that $\G$ counts holomorphic curves in $X\setminus Q$ with 
negative punctures asymptotic to $Q$, whereas $\A$ counts holomorphic
curves in $X$ with boundary on $Q$ (both without positive punctures). So
Theorem~\ref{thm:twisted-iso} gives a precise formulation of the 
principle that ``holomorphic curves with punctures asymptotic to a
Lagrangian submanifold $Q$ carry the same information as holomorphic
curves with boundary on $Q$''. 

{\small 
\subsubsection*{Technical remarks}

Similarly to Theorem~\ref{thm:master-L}, Theorem~\ref{thm:master-L2}
(and hence also Theorem~\ref{thm:twisted-iso}) is conjectural at this
point, with roughly the same analytic difficulties for the proof. As
new phenomena, one now also has disk and sphere bubbling. 

While the parity of the variables can be given precise meaning, the
integer grading depends on the choice of capping surfaces used, and
changing these capping surfaces shifts degrees by an appropriate
(even) Maslov index, cf. \cite{EGH}. 

To make the discussion more precise, one would also have to keep track
of certain Novikov conditions as described in \cite{EGH}.
}

\newpage
\appendix
\renewcommand{\thesection}{Appendix A}
\section{Towards relative SFT\\ (joint with
K.~Mohnke)}\label{sec:rel} 
\renewcommand{\thesection}{A}

In this section we argue that string topology should play an essential
role in relative symplectic field theory. 
This 
appendix is the result of joint work of the authors with K.~Mohnke. 

Let $(Y^{2n-1},\lambda)$ be a (not necessarily compact) contact
manifold which is convex in the sense of~\cite{EGH}. For simplicity,
we assume that there are no closed Reeb orbits. E.g.~this is the case
for $\R^{2n-1}$ or, more generally, any 1-jet space with its canonical
contact form. Let $\Lambda\subset Y$ be a compact Legendrian
submanifold, which for simplicity we assume to have zero Maslov
class. 

Let us first recall the definition of relative contact homology as
in~\cite{Ch, EGH, EES, EES:2}. A {\em Reeb chord} is a Reeb orbit starting
and ending on $\Lambda$. We assume that all Reeb chords are
nondegenerate.  
Let $J$ be a cylindrical almost complex structure on $\R\times Y$
adjusted to $\lambda$.
For a cyclically ordered collection of Reeb chords
$\Gamma=(\gamma_1,\dots,\gamma_s)$ 
denote by 
$\MM(\Gamma)$ 
the moduli space of $J$-holomorphic disks in $\R\times Y$ with
boundary on $\R\times\Lambda$ and $s$ boundary punctures (in
counterclockwise order) asymptotic to the $\gamma_i$. Here the
punctures can be positive or negative in any order.  
The vanishing of the Maslov class allows us to assign
to each Reeb chord $\gamma$ a Conley-Zehnder index
$\CZ(\gamma)\in\Z$ such that 
the following dimension formula holds:
$$
   \dim\MM(\Gamma) = (n-3)(1-s^+) +
   \sum_{i\in I^+}\bigl(\CZ(\gamma_i)-1\bigr) -
   \sum_{j\in I^-}\bigl(\CZ(\gamma_j)-1\bigr).  
$$
Here $I^\pm\subset\{1,\dots,s\}$ are the sets of indices corresponding
to positive resp.~negative punctures and $s^+=|I^+|$. 
If this dimension is $1$ denote by $n(\Gamma)\in\Q$ the
algebraic count of points in $\MM(\Gamma)/\R$. 

Assign to each Reeb chord $\gamma$ a formal variable $q_\gamma$ of
degree 
$$
   |q_\gamma| := \CZ(\gamma)-1
$$
and let $\AA$ be the free graded commutative algebra in the
$q_\gamma$. Define a derivation $\p:\AA\to\AA$ of degree $-1$ on
generators by
$$
   \p q_\gamma :=
   \sum_{\gamma_1,\dots,\gamma_k}n(\gamma_1,\dots,\gamma_k\gamma) 
   q_{\gamma_1}\cdots q_{\gamma_k},
$$
where the $\gamma$ corresponds to a positive puncture and
$\gamma_1,\dots,\gamma_k$ correspond to negative punctures. 
The following result is proved in~\cite{EES:2} for 1-jet spaces and used
to distinguish Legendrian submanifolds in all dimensions.

\begin{thm}
We have $\p^2=0$ and the {\em relative contact homology} $H_*(\AA,\p)$
is an invariant of $\Lambda$ up to Legendrian isotopy. 
\end{thm}

Now let us turn to disks with more than one positive puncture. By
analogy with absolute SFT, one would expect their counts to define
operations on relative contact homology. For example, denote by
$\mu:\AA\otimes\AA\to\AA$ the operation obtained by counting rigid
holomorphic disks in $\R\times Y$ with boundary on $\R\times\Lambda$
and two neighbouring positive punctures and an arbitrary number of
negative punctures on the boundary. Figure~\ref{fig:18} shows the
codimension one degenerations of a 1-dimensional 
moduli space of disks with two positive boundary punctures. 
\begin{figure}[h]
\begin{center}
\epsfbox{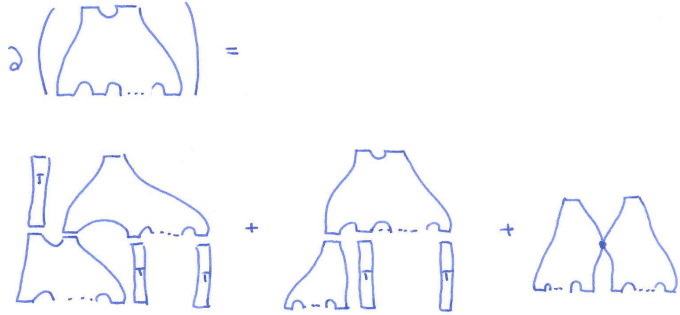}
\caption{Codimension 1 phenomena in a moduli space of disks with two
  positive punctures}
\label{fig:18}
\end{center}
\end{figure}

Without the last term, Figure~\ref{fig:18} would show that
$\mu\p+\p\mu=0$ and $\mu$ descends to $H_*(\AA,\p)$. However, this
conclusion does {\em not} hold due to the presence of the last term
and in general $\mu$ does {\em not} descend to an operation on
relative contact homology. So how do we proceed from here?

One way is based on the observation that the last term is the string
bracket (in $\Lambda$) 
of two moduli spaces with one positive
boundary puncture. Moreover, since the string bracket has degree
$3-n$, this term will for $n\geq 4$ involve moduli spaces of higher
dimensions. Motivated by the previous sections, let us assemble the
moduli spaces of {\em all} dimensions of holomorphic disks with
arbitrary numbers of positive and negative boundary punctures into a
{\em Hamiltonian} $\h$. Here the negative punctures correspond to the
variables $q_\gamma$ and the positive punctures to variables
$p_\gamma$ of degree
$$
   |p_\gamma| := n-2-\CZ(\gamma).
$$
For each holomorphic disk, the evaluation at the boundary gives a
broken closed string in $\Lambda$, which at the break points jumps
from one end point of a Reeb chord to another. 
Thus $\h$ belongs to the space $\PP$ of power series in the
non-commuting variables $p_\gamma$, $q_\gamma$ (finite in the
$q_\gamma$), with coefficients in the chains in the space of broken
closed strings, modulo cyclic permutations acting diagonally,
i.e. both on the chains and on the $p$'s and $q$'s. Consider the
following operations on the space $\PP$:
\begin{itemize}
\item the singular boundary operator $\p$;
\item the string bracket $[\ ,\ ]_\str$;
\item the SFT bracket $[\ ,\ ]_\SFT$ given by deleting a pair
  $p_\gamma,q_\gamma$ and gluing the broken closed strings along the
  corresponding Reeb chord;
\item the operation $\delta$ given by the action of a constant disk
  with one positive and one negative puncture at the same Reeb orbit
  $\gamma$ and with a fixed marked point on its boundary on the chains
  of broken closed strings by the loop-string-product as in
  Figure~\ref{fig:19}. 
\end{itemize}
\begin{figure}[h]
\begin{center}
\epsfbox{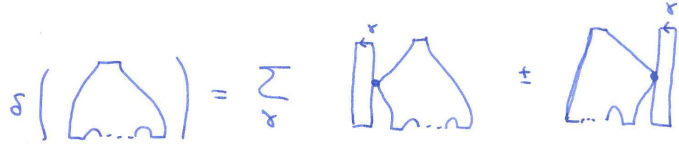}
\caption{The operation $\delta$}
\label{fig:19}
\end{center}
\end{figure}

\begin{conj}\label{conj:master-rel}
$[\ ,\ ]:=[\ ,\ ]_\SFT+[\ ,\ ]_\str$ defines a Lie bracket of degree
$3-n$ and $\p+\delta$ defines an operator of square zero and degree
$-1$ which is a derivation of this bracket. Moreover, the Hamiltonian
$\h$ satisfies the equation 
\begin{equation}\label{eq:master-rel}
   (\p+\delta)\h + \frac{1}{2}[\h,\h] = 0.
\end{equation}
The homology $H_*(\PP,\p+\delta+[\h,\cdot])$ is an invariant of
$\Lambda$ up to Legendrian isotopy. 
\end{conj}

We see that, in contrast to the absolute case, in relative SFT string
topology already enters in the definition of the theory. 

For simplicity, we have formulated Conjecture~\ref{conj:master-rel} for
holomorphic disks and in the absence of closed Reeb orbits. For the
full relative SFT one should consider arbitrary compact Riemann
surfaces with boundary and positive and negative punctures in the
interior as well as on the boundary. The algebraic formulation,
analytical foundations, and applications of this theory are the
subject of current research. 

{\small 
\subsubsection*{Technical remarks}


In order to give Equation~\eqref{eq:master-rel} rigorous meaning one must
define the string operations on the chain level, with all the
technical difficulties this entails. For the invariance proof, one
also needs a discussion of the morphism induced by moduli spaces of
curves in a cobordism and of homotopies between such morphisms.
}



\begin{thebibliography}{999}

\bibitem{AS} A.~Abbondandolo and M.~Schwarz, {\em On the Floer homology of
    cotangent bundles}, Comm. Pure Appl. Math. {\bf 59}, 254--316 (2006)

\bibitem{BEHWZ} F.~Bourgeois, Y.~Eliashberg, H.~Hofer, K.~Wysocki and
  E.~Zehnder, {\em Compactness results in symplectic field theory},
  Geom. \& Topol.~7  (2003), 799--888   

\bibitem{BM} F.~Bourgeois and K.~Mohnke, {\em Coherent orientations in
    symplectic field theory}, Math.~Z. {\bf 248}, 123--146 (2004).




\bibitem{Chas} M.~Chas, {\em Combinatorial Lie bialgebras of curves on
    surfaces}, Topology 43, 2004, 543--568

\bibitem{CS} M.~Chas and D.~Sullivan, {\em String topology}, preprint 
math.GT/9911159 (1999).

\bibitem{CS2} M.~Chas and D.~Sullivan, {\em Closed string operators in 
    topology leading to Lie bialgebras and higher string algebra},
    math.GT/0212358.

\bibitem{Ch} Y.~Chekanov, {\em Differential algebra of Legendrian
  links}, Invent.~Math.~150  (2002),  no.~ 3, 441--483

\bibitem{CFL} K.~Cieliebak, K.~Fukaya and J.~Latschev, in preparation

\bibitem{CL:exact} K.~Cieliebak and J.~Latschev, {\em Symplectic field
  theory and string topology: the exact case}, in preparation.

\bibitem{CL:string} K.~Cieliebak and J.~Latschev, {\em Chain level
    string topology}, in preparation.

\bibitem{CM-comp} K.~Cieliebak and K.~Mohnke, {\em Compactness for
    punctured holomorphic curves}, Conference on Symplectic Topology,
    J. Symplectic Geom.  3  (2005),  no. 4, 589--654 

\bibitem{CM-trans} K.~Cieliebak and K.~Mohnke, {\em Symplectic hypersurfaces 
    and transversality in Gromov-Witten theory}, arXiv:math/0702887. 

\bibitem{CMS} K.~Cieliebak, I.~Mundet i Riera and D.~Salamon, {\em
    Equivariant moduli problems, branched manifolds, and the Euler
    class}, Topology~42 (2003), no.~3, 641--700


\bibitem{EES} T.~Ekholm, J.~Etnyre and M.~Sullivan, {\em The contact
    homology of Legendrian submanifolds in $\R^{2n+1}$},
    J.~Diff.~Geom.  71  (2005),  no. 2, 177--305

\bibitem{EES:2} T.~Ekholm, J.~Etnyre and M.~Sullivan, {\em Legendrian
    contact homology in $P\x \R$},  math/0505451

\bibitem{EGH} Y.~Eliashberg, A.~Givental and H.~Hofer, {\em Introduction to
    symplectic field theory}, GAFA 2000 Visions in Mathematics special volume,
    part II, 560--673

\bibitem{El} Y.~Eliashberg, {\em Symplectic field theory and its
    applications}, talk at the ICM 2006. 

\bibitem{Fu:06} K.~Fukaya, {\em Application of Floer homology of
    Lagrangian submanifolds to symplectic topology}, in: Morse
    theoretic methods in nonlinear analysis and in symplectic
    topology, ed. by P.~Biran, O.~Cornea and F.~Lalonde, 231--276

\bibitem{FOOO} K.~Fukaya, Y.~G.~Oh, H.~Ohta and K.~Ono, {\em Lagrangian
    intersection Floer homology - anomaly and obstruction}, preprint,
    2000, revision 2006


\bibitem{Ge:1} E.~Getzler, {\em Batalin-Vilkovisky algebras and
    two-dimensional topological field theories},
    Comm.~Math.~Phys.~{\bf 159}, 265--285 (1994). 



\bibitem{Go} W.~M.~Goldman, {\em Invariant functions on Lie groups and
    Hamiltonian flows of surface group representations},
    Invent.~Math.~85 (1986), 263--302

\bibitem{HWZ} H.~Hofer, K.~Wysocki and E.~Zehnder, {\em Polyfolds and
    Fredholm Theory I-V}, in preparation. 




\bibitem{Kr} O.~Kravchenko, {\em Deformations of Batalin-Vilkovisky
    algebras}, Poisson geometry (Warsaw, 1998), 131--139, Banach Center
   Publ., 51, Polish Acad. Sci., Warsaw, 2000






\bibitem{Ra} H.-B.~Rademacher, {\em On the average indices of closed
    geodesics}, J.~Diff.~Geom. {\bf 29}, 65--83 (1989). 

\bibitem{SW} D.~Salamon and J.~Weber, {\em Floer homology and the heat
    flow}, Geom.~Funct.~Anal.~16 (2006), no.~5, 1050--1138  




\bibitem{TT} D.~Tamarkin and B.~Tsygan, {\em Noncommutative
    differential calculus, homotopy BV algebras and formality
    conjectures}, Methods Funct. Anal. Topology 6 (2000), no. 2, 85--100 

\bibitem{Tu} V.~G.~Turaev, {\em Skein quantization of Poisson algebras
    of loops on surfaces}, Ann.~Sci.~Ecole Norm.~Sup. (4) 24 (1991),
    635--704

\bibitem{Vi} C.~Viterbo, {\em Functors and computations in Floer
    homology with applications}, preprint 1998. 


\end{thebibliography}
\end{document}